\documentclass[aap]
{imsart}

\RequirePackage{amsthm,amsmath,amsfonts,amssymb,mathrsfs,color}
\usepackage{stmaryrd} 
\RequirePackage[numbers]{natbib}

\startlocaldefs

\newcommand{\R}{{\mathbb R}}
\newcommand{\I}{{\mathbb I}}
\newcommand{{\tildeGamma}}{g_1}

\def\leftB{[\![}
\def\rightB{]\!]}

\newcommand{\E}{{\mathbb E}}\makeatletter
\newcommand{\N}{{\mathbb N}}

\theoremstyle{definition}
\newtheorem*{condition*}{\protect\conditionname}
\theoremstyle{plain}
\newtheorem{thm}{\protect\theoremname}[section]
\theoremstyle{plain}
\newtheorem{prop}[thm]{\protect\propositionname}
\theoremstyle{definition}

\theoremstyle{plain}
\newtheorem{lem}[thm]{\protect\lemmaname}
\theoremstyle{plain}

\theoremstyle{plain}
\newtheorem{remark}[thm]{\protect\remarkname}
\theoremstyle{plain}

\theoremstyle{plain}
\newtheorem*{assumption*}{\protect\assumptionname}
\numberwithin{equation}{section}
\makeatother
\providecommand{\assumptionname}{Assumption}
\providecommand{\conditionname}{Condition}
\providecommand{\corollaryname}{Corollary}
\providecommand{\definitionname}{Definition}
\providecommand{\lemmaname}{Lemma}
\providecommand{\propositionname}{Proposition}
\providecommand{\theoremname}{Theorem}
\providecommand{\remarkname}{Remark}
\providecommand{\examplename}{Example}

\endlocaldefs

\begin{document}

\begin{frontmatter}
\title{Convergence rate of the Euler-Maruyama scheme applied to diffusion processes with $L^q-L^\rho $  drift coefficient and additive noise}
\runtitle{ Euler scheme  with $L^q-L^\rho $  drift and additive noise}

\begin{aug}
\author[A]{\fnms{Benjamin} \snm{Jourdain}\ead[label=e1]{benjamin.jourdain@enpc.fr}} and
\author[B]{\fnms{St\'ephane} \snm{Menozzi}\ead[label=e2,mark]{stephane.menozzi@univ-evry.fr}}
\address[A]{Cermics, Ecole des Ponts, INRIA, Marne-la-Vall\'ee, France, \printead{e1}}

\address[B]{Laboratoire de Mod\'elisation Math\'ematique d'Evry (LaMME), Universit\'e d'Evry Val d'Essonne, Universit\'e Paris Saclay, UMR CNRS 8071, 23 Boulevard de France, 91037 Evry, France and Laboratory of Stochastic Analysis, HSE University, Moscow, Pokrovsky Boulevard, 11, Russian Federation, \printead{e2}}
\end{aug}

\begin{abstract}
We are interested in the time discretization of 
 stochastic differential equations with additive $d$-dimensional Brownian noise and $L^q-L^\rho $  drift coefficient when the condition $\frac d\rho+\frac 2q<1$, under which Krylov and R\"ockner \cite{kryl:rock:05} proved existence of a unique strong solution, is met. We show weak convergence with order $\frac 12(1-(\frac d\rho+\frac 2q))$ which corresponds to half the distance to the threshold for the Euler scheme with randomized time variable and cutoffed drift coefficient  so that its contribution on each time-step does not dominate the Brownian contribution. More precisely, we prove that both the diffusion and this Euler scheme admit transition densities and that the difference between these densities is bounded from above by the time-step to this order multiplied by some centered Gaussian density. 
\end{abstract}

\begin{keyword}[class=MSC2020]
\kwd[Primary ]{ 60H35, 60H10}
\kwd[; secondary ]{65C30, 65C05}
\end{keyword}

\begin{keyword}
\kwd{diffusion processes, singular drift, Euler scheme, weak error analysis}
\end{keyword}

\end{frontmatter}


\section{Introduction}

In the present paper, we are interested in the Euler-Maruyama discretization of the stochastic differential equation
\begin{equation}
  X_t = x + W_t + \int_0^t b(s,X_s)\,ds,\quad t\in[0,T],\label{eds}
\end{equation}
where $x\in\R^d$, $(W_t)_{t\ge 0}$ is a $d$-dimensional Brownian motion on some filtered probability space $(\Omega,\mathcal F, (\mathcal F_t)_{t\ge 0},\mathbb P) $, $T\in(0,+\infty)$ is a finite time horizon and the drift coefficient $b:[0,T]\times\R^d\to\R^d$ is measurable and satisfies the integrability condition : $\|b\|_{L^{q}([0,T] ,L^{\rho}(\R^d))}=:\|b\|_{L^{q}-L^{\rho}}<\infty$  for some $\rho,q>0$ such that
\begin{equation}
\label{COND_KR}
\rho\ge 2\mbox{ and }\frac d\rho+\frac 2q<1, 
\end{equation}
which clearly implies that $\rho>d$ and $q>2$. When $\rho$ and $q$ are both finite,
$$\|b\|_{L^{q}-L^{\rho}}=\left(\int_0^T\left(\left(\int_{\R^d}|b(t,y)|^\rho dy\right)^{1/\rho}\right)^qdt\right)^{1/q}$$
and when $\rho=+\infty$ then $\left(\int_{\R^d}|b(t,y)|^\rho dy\right)^{1/\rho}$ is replaced by the essential supremum of $y\mapsto |b(t,y)|$ with respect to the Lebesgue measure on $\R^d$ while, when $q=+\infty$, the power $1/q$ of the integral of the $q$-th power of the function of the time variable over $[0,T]$ is replaced by the essential supremum of this function with respect to the Lebesgue measure on $[0,T]$. 
This framework was introduced by Krylov and R\"ockner \cite{kryl:rock:05}, who established strong existence and uniqueness for the above equation under the integrability condition \eqref{COND_KR}. The critical case has recently been treated by Krylov \cite{kryl:20} and R\"ockner and Zhao \cite{rock:zhao:20}, \cite{rock:zhao:21} who respectively addressed the strong well-posedness of \eqref{eds} in the time-homogeneous case (then $q=+\infty$) when $\rho=d$, and the weak and strong well-posedness when $ \frac{d}{\rho}+\frac{2}q=1$ for a time dependent drift coefficient.

Let us emphasize that  dynamics of type \eqref{eds} appear in many applicative fields.  
In \cite{kryl:rock:05}, the authors discussed the connection with some models arising from statistical mechanics or interacting particle systems, see \cite{albe:kond:rock:03}. Singular kernels appear as well in several domains 
related to mathematical physics like fluid dynamics or electro-magnetism. We can for instance mention the Biot-Savart kernel behaving in $y/|y|^{d} $  near the origin. A similar singularity  also appears in the parabolic elliptic Keller-Segel equation. Let us emphasize that for such singularity, the integrability conditions \eqref{COND_KR} are not met. However,  in dimension $d=2$, a kernel behaving around 0 as $ |y|^{\varepsilon+1-d}$, $\varepsilon>0$ could be considered. In this last setting we can refer e.g. to the work by Jabin and Wang \cite{jabi:wan:18} for related applications.  We can eventually quote the important work of Zhang and Zhao \cite{zhan:zhao:21} who established existence of a stochastic Lagragian path for a Leray solution of the 3d Navier-Stokes equation. In that case, $d=3$, $\rho=2 $, $q=\infty $ so condition \eqref{COND_KR} is not met but the drift has some additional properties, namely its gradient also belongs to $L^2-L^2$.

It is therefore important to address the question of the approximation of \eqref{eds}. To this end, the easiest, and maybe the most natural at first sight, way consists in introducing the Euler-Maruyama scheme with step $h>0$. Anyhow, in the current singular context it needs to be tailored appropriately.
Namely, we consider  a \textit{cutoff} with order related to the singularity of the drift. The  coefficient with cutoff is defined by
\begin{equation}
   b_h(t,y)=\I_{\{|b(t,y)|>0\}}\frac{|b(t,y)|\wedge (Bh^{-(\frac 1q+\frac d{2\rho})})}{|b(t,y)|}b(t,y),\;(t,y)\in[0,T]\times\R^d\label{cutoffb}
\end{equation}
for some constant $B\in (0,+\infty)$. When $b\in L^\infty-L^\infty$, we choose $B=\|b\|_{L^\infty-L^\infty}$ so that $b_h=b$ for each step $h=T/n$, $n\in\N^*$. Since, according to \eqref{COND_KR}, $\frac 1q+\frac d{2\rho}<\frac 1 2$, the contribution of the cutoffed drift on each time step does not dominate the Brownian contribution.

Furthermore, to get rid of any assumption stronger than mere measurability (and integrability) concerning the regularity of the drift coefficient with respect to the time variable, we choose to randomize the time variable.

   The time randomization relies on independent random variables $(U_k)_{k \in \left\llbracket0,n-1 \right\rrbracket}$, where from now on we will denote by $\leftB\cdot,\cdot\rightB $ the integer intervals, which are respectively distributed according to the uniform law on $[kh,(k+1)h]$ and independent from $(W_t)_{t\ge 0}$. Notice that  this sequence is of course not needed when the drift coefficient is time-homogeneous.
The resulting scheme is initialized by $X^h_0=x$ and evolves inductively on the regular time-grid $(t_k=kh)_{k \in \left\llbracket0,n\right\rrbracket}$ with $h=\frac T n$ by:
  \begin{equation}
    X^{h}_{t_{k+1}}=X^h_{t_k}+\left(W_{t_{k+1}}-W_{t_k}\right)+b_h\left(U_k, X^h_{t_k}\right)h
    . \label{euler}
  \end{equation}

 

  We then consider the following continuous time interpolation of the scheme:
  \begin{equation}
    X^h_t = x + W_t + \int_0^t b_h\left(U_{{\lfloor \frac s h\rfloor}},X^h_{\tau^h_s}\right)ds,\quad t\in[0,T] \mbox{ where } \tau^h_s  =\lfloor s/h\rfloor h 
    \label{conteuler}.
  \end{equation}
The cutoff threshold in \eqref{cutoffb} permits to get rid of the drift in the Gaussian estimates that we will derive for the transition densities of the scheme : for all $s\in (t_k,t_{k+1}] $ and $y\in \R^d $,
$$\exp\left(c\frac{|b_h(t,y)(s-t_k)|^2}{s-t_k}\right)\le \exp(cB^2 h^{1-(\frac 2q+\frac d\rho)}) \overset{\eqref{COND_KR}}{\underset{h\rightarrow 0}{\longrightarrow}} 1. $$
For this sole purpose, the natural threshold would have been in $h^{-1/2} $ rather than $h^{-(\frac 1q+\frac d{2\rho})}$. The interest of the stronger cutoff is that it also permits to control in the proof of Theorem \ref{MTHM}  the error on the first time-step when the drift coefficient is computed at the deterministic initial position $x$, while it is computed at positions with densities satisfying some Gaussian estimates at the subsequent steps. The error bound  of Theorem \ref{MTHM} remains valid for the $h^{-\frac 12}$ cutoff scale provided we set the drift to zero on the first time step. The alternative scheme writes:
\begin{equation}
   \bar  X^{h}_{t_{k+1}}=\bar X^h_{t_k}+\left(W_{t_{k+1}}-W_{t_k}\right)+\bar b_h\left(U_k, \bar X^h_{t_k}\right)h
    ,\label{euler_BIS}
  \end{equation}
 with  \begin{equation}
   \bar b_h(t,y)=\I_{\{t\ge h,|b(t,y)|>0\}}\frac{|b(t,y)|\wedge (Bh^{-\frac 12})}{|b(t,y)|}b(t,y),\;(t,y)\in[0,T]\times\R^d.\label{cutoffb_BIS}
\end{equation}
It can be  convenient to choose one scheme or the other. The first one writes as a usual discretization scheme, but needs the drift $b(s,x)$ to be defined for all $x\in\R^d$ and almost any $s\in [0,h]$.
This is not too restrictive if one has in mind some physical models for which the singularities are precisely located at some points in space, where it would actually suffice to assign the drift an arbitrary prescribed value.
On the other hand, from the theoretical viewpoint, the scheme \eqref{euler_BIS} allows to simply consider a class of functions in $L^q-L^\rho$. 
Eventually, let us stress that, apart from the contributions of the cutoff error and the first time step (see in particular the analysis of the terms $\Delta_t^2 $ and $\Delta^5_t$ in Section \ref{THE_SEC_FOR_PROOF_MTHM}), the choice of the scheme has a minimal impact on the proof of the error estimation since both dynamics \eqref{euler} and \eqref{euler_BIS} satisfy the Gaussian estimates of Proposition \ref{EST_DENS_SCHEME}.
\\

While the convergence properties of the Euler-Maruyama scheme are well understood for SDEs with smooth coefficients, the case of irregular coefficients is still an active field of research.
Concerning the strong error, the additive noise case is investigated in \cite{HalKloe} where Halidias and Kloeden only prove convergence and in \cite{DarGer,NeuSzo} where rates are derived. Dareiotis and Gerencs\'er \cite{DarGer} obtain root mean square convergence with order $1/2-$ (meaning $1/2-\varepsilon$ for arbitrarily small $\varepsilon >0$) in the time-step for bounded and Dini-continuous time-homogeneous drift coefficients and check that this order is preserved in dimension $d=1$ when the Dini-continuity assumption is relaxed to mere measurability. In the scalar $d=1$ case, Neuenkirch and Sz\"olgyenyi \cite{NeuSzo} assume that the drift coefficient is the sum of a $ \mathcal{C}^{2}_b$ part and a bounded integrable irregular part with a finite Sobolev-Slobodeckij semi-norm of index $\kappa \in (0,1)$. They prove root mean square convergence with order $\frac{3}{4}\wedge\frac{1 + \kappa}{2} -$ for the equidistant Euler-Maruyama scheme, the cutoff of this order at $\frac{3}{4}$ disappearing for a suitable non-equidistant time-grid.
Note that an exact simulation algorithm has been proposed by \'Etor\'e and Martinez \cite{EtoMArt} for one-dimensional SDEs with additive noise and time-homogeneous and smooth except at one discontinuity point drift coefficient.
\\More papers have been devoted to the strong error of the Euler scheme for SDEs with a non constant diffusion coefficient : \cite{GyoKryl,Gyo,Yan,GR,NgTag1,NgTag2,BHY}. Recent attention has also been paid to the Euler-Maruyama discretization of SDEs with a piecewise Lipschitz drift coefficient and a globally Lipschitz diffusion coefficient which satisfies some non-degeneracy condition on the discontinuity hypersurface of the drift coefficient : \cite{LeoSzo1,LeoSzo2,LeoSzo3,MGYaro1,NeunSzoSpr}.

  We will here focus on the so-called \textit{weak error} between the diffusion and the Euler scheme \eqref{euler}, namely the quantity
$${\mathscr E}(x,T,\varphi,h):=\E_x[\varphi(X_T^{h})]-\E_x[\varphi(X_T)],$$
for a suitable class of test functions $\varphi $ which can even be a Dirac mass. In the additive noise case considered in the present paper, Kohatsu-Higa, Lejay and Yasuda \cite{KohY}, prove that for $\varphi$ thrice continuously differentiable with polynomially growing derivatives, the convergence holds with order $1/2-$ when $d\ge 2$ (resp. $1/3-$ when $d=1$) and the drift coefficient is time homogeneous, bounded and Lipschitz except on a set $G$ such that $\varepsilon^{-d}$ times the Lebesgue measure of $\{x\in\R^d:\inf_{y\in G}|x-y|\le\varepsilon\}$ is bounded. Suo, Yuan and Zhang \cite{SuYuZ} prove convergence in total variation with order $\frac{\alpha}{2}$ for time-homogeneous drift coefficients with at most linear growth and satisfying an integrated against some Gaussian measure $\alpha$-H\"older type regularity condition. 
\\

In the much more general multiplicative noise setting, when the diffusion and drift coefficients are smooth, from the seminal work of Talay and Tubaro \cite{tala:tuba:90} to the extensions to the hypoelliptic setting, see e.g. the works by  Bally and Talay \cite{ball:tala:96:1}, \cite{ball:tala:96:2}, it has been established that the above weak error 
(with $b\left(t_k,X^h_{t_k}\right)$ replacing $b_h\left(U_{k}, X^h_{t_k}\right)$ in the right-hand side of \eqref{euler}) has order one w.r.t. the discretization parameter $h$. When $\varphi $ is a Dirac mass we can also refer to \cite{kona:mamm:02} or to \cite{MenoKona}  where non-degenerate bounded H\"older coefficients are considered (see also \cite{NFrik} in the framework of skew diffusions). The common point in all these results is the key role played by the Feynman Kac partial differential equation (PDE) associated with \eqref{eds} which permits to write the error as the expectation of a time integral of the sum of terms with derivatives of the solution to this PDE multiplied by the difference between the drift and squared diffusion coefficients at the current time and position of the Euler scheme and at the last discretization time and corresponding position. This permits to exploit the regularity of these coefficients to derive the order of convergence.

It is however clear that for \textit{rough} coefficients, like in the current $L^q-L^\rho$ framework, another strategy is needed. For a bounded measurable drift ($\rho=q=\infty$), a new idea was proposed in \cite{benc:jour:20} consisting in comparing the expansions of the densities at time $T$ of the diffusion and its Euler scheme with randomized time variable along the solution of the heat equation (in place of the Feynman-Kac PDE) with terminal condition $\varphi$ equal to $\delta_y(dz)$. This solution is $(s,z)\mapsto g_1(T-s,y-z)$ where $g_1(t,.)$ denotes the Brownian density at time $t>0$.

We will check in Propositions \ref{THM_HK_CONT} and \ref{EST_DENS_SCHEME} that both the SDE \eqref{eds} and the scheme \eqref{conteuler} admit transition densities which can be expanded around this Gaussian density as expected from a formal application of It\^o's formula. More precisely, for $s\in[0,T)$ and $x\in\R^d$, the solution to \begin{equation}
  dX_t=dW_t+b(t,X_t)dt\label{sdediff}\end{equation} started from $x$ at time $s$ admits at time $t\in (s,T]$ a density with respect to the Lebesgue measure on $\R^d$ denoted by $y\mapsto \Gamma(s,x,t,y)$ and, as expected formally by computing $dg_1(T-s,y-X_s)$ by It\^o's formula and taking expectations, 
\begin{align*}
\forall y \in \R^d, \;\Gamma(0,x,T,y)
  =g_1(T,y-x)-\int_{0}^{T}\E\left[b(s,X_{s})\cdot\nabla_y g_1(T-s,y-X_s)\right]ds.
  \end{align*}
In a similar way, for $k\in\left\llbracket0,n-1 \right\rrbracket$ and $x\in\R^d$, the solution to \begin{equation}
   dX^h_t=dW_t+b_h(U_{\lfloor \frac t h\rfloor},X^h_{\tau^h_t})dt\label{eulerdiff}\end{equation} (resp. the same dynamics with $b_h$ replaced by $\bar b_h$) started from $x$ at time $t_k$ admits at  time $t\in (t_k,T]$ a density with respect to the Lebesgue measure on $\R^d$ denoted by $y\mapsto \Gamma^h(t_k,x,t,y)$ (resp. $y\mapsto \bar\Gamma^h(t_k,x,t,y)$) and 
\begin{align}
 \forall y \in \R^d, \;\Gamma^h(0,x,T,y)
  =g_1(T,y-x)-\int_{0}^{ T}\E\left[b_h(U_{\lfloor \frac sh\rfloor },X^h_{\tau_s^h})\cdot\nabla_y g_1(T-s,y-X^h_s)\right]ds
\end{align}
(resp. the same equation holds with $\Gamma^h$ and $b_h$ replaced by $\bar\Gamma^h$ and $\bar b_h$).

Taking the difference of the two expansions, we obtain
\begin{align}
\Gamma^h(0,x,T,y)&-\Gamma(0,x,T,y)\notag\\
=&\int_{0}^T ds[\Gamma(0,x,s,z)-\Gamma^h(0,x,s,z)]b(s,z)\cdot\nabla_y g_1(T-s,y-z)dz \notag\\
  &+\int_{0}^T ds\int_{\R^d}\Gamma^h(0,x,s,z)(b(s,z)-b_h(s,z))\cdot\nabla_y g_1(T-s,y-z)dz\notag\\&+\int_{0}^T ds\int_{\R^d}[\Gamma^h(0,x,s,z)-\Gamma^h(0,x,\tau^h_s,z)]b_h(s,z).\nabla_y g_1(T-s,y-z) dz\notag\\
&+\E\bigg[\int_{0}^T dsb_h(U_{\lfloor s/h\rfloor},X_{\tau^h_s}^h)\cdot (\nabla_y g_1(T-U_{\lfloor s/h\rfloor},y-X_{\tau^h_s}^h)\notag\\
&\hspace*{3cm}-\nabla_y g_1(T-s,y-X_s^h))\bigg].
\label{THIRD_EXP_ERR}
\end{align}

This formula actually emphasizes that, in order to give a convergence rate for the Euler approximation, two preliminary results are needed:
\begin{trivlist}
\item[-] estimations on the heat kernel $\Gamma^h$ of the Euler scheme in order to deal with the second (cutoff error) and fourth terms in the right-hand side,
\item[-] estimations of its H\"older modulus w.r.t. to the forward time variable to deal with the third term in the right-hand side.
\end{trivlist}

These properties are established in Proposition \ref{EST_DENS_SCHEME} below using an approach inspired from \cite{meno:pesc:zhan:21}. The first term in the right-hand side will be treated through a Gronwall type argument.

 In \cite{benc:jour:20}, for $q=\rho=+\infty $, starting from a similar decomposition (actually the term involving the modulus of the heat kernel with respect to the forward time variable is written with the transition density of the diffusion instead of that of the Euler scheme), the authors derived a convergence rate of order $1/2$ w.r.t. $h$  for the total variation distance between the law of the diffusion and its Euler scheme for a bounded drift. 

 In our main result, we extend this estimation in a precised way to the case $b\in L^q-L^\rho$ with $\frac{d}{\rho}+\frac{2}{q}\in (0,1)$ by showing that the difference between the densities is bounded from above by $C h^{\frac{1}{2}\left(1-\left(\frac{d}{\rho}+\frac 2q\right)\right)}$ multiplied by some centered Gaussian density. In fact, in the case $q=\rho=+\infty$ for both discretization schemes and when $(d,q)=(1,\infty)$ for the scheme \eqref{conteuler}, the estimation is perturbed by an extra logarithmic factor.
 
 \begin{thm}[Convergence Rate for the Euler-Maruyama approximation with $L^q-L^\rho $ drift]\label{MTHM}
 Assume that \eqref{COND_KR} holds. Set:
$$\alpha:= 1-\left(\frac{d}{\rho}+\frac 2q\right).$$
 Then,  for all $c>1 $ there exists a constant $C_{c}<\infty$ s.t. for  all $h=T/n$ with $n\in\N^*$, and all $t\in(0,T]$, $x,y\in \R^d $
 \begin{align*}
   |\Gamma^h(0,x,t,y)-\Gamma(0,x,t,y)| &\le C_{c}  h^{\frac{\alpha} 2}\left(1+(\I_{\{\alpha=1\}}+\I_{\{(d,q)=(1,\infty)\}})\ln n\right)g_c(t,y-x),\\
   |\bar\Gamma^h(0,x,t,y)-\Gamma(0,x,t,y)| &\le C_{c}  h^{\frac{\alpha} 2}\left(1+\I_{\{\alpha=1\}}\ln n\right)g_c(t,y-x).
 \end{align*}
where $g_c(u,\cdot)$ stands for the density of the centered Gaussian vector in dimension $d$ with covariance matrix $ cuI_d,\ u>0$.
\end{thm}

%

  \begin{remark}[About the positive exponent for the time in the error]
 We point  out that, since we are handling rough drifts, and therefore cannot proceed with the expansions of the heat-kernels beyond orders greater than 2, there is no time singularity appearing in the final bound for the error.
We can refer to  \cite{kona:mamm:02}, \cite{guyo:06} or more recently \cite{gobe:laba:08} for a specific description of the time-singularity for  the error expansion when the coefficients are smooth. In that case, the convergence rate is $h$ and the best upper bound for the time singularity comes from \cite{gobe:laba:08} and has order $t^{-1/2} $.

 \end{remark}

 \begin{remark}[About the convergence rate for smoother coefficients]Let us mention that the proof suggests that for drifts H\"older continuous with exponent $\beta$ with respect to the spatial variable and with exponent $\frac\beta 2$ with respect to the time variable, the convergence rate of the Euler approximation for an additive noise should be in $h^{\frac 12+\frac \beta 2}$ if $\beta\in (0,1) $ and not in $h^{\frac \beta 2} $ as established  in e.g. \cite{MikP}, \cite{MenoKona} in which a multiplicative noise was anyhow also taken into consideration. We plan to investigate this question in a future work. We can mention the recent work by Dareiotis \textit{et al.} \cite{dare:gere:le:21}, which investigates the strong error for the Euler scheme with rough drifts, namely quantities of the form $\E[\sup_{s\in [0,T]}|X_s-X_s^h|^p]^{\frac 1p} $, and in which are also obtained error bounds of this order (see Theorem 1.5 therein which precisely gives this convergence rate for a bounded drift in the homogeneous Sobolev-Slobodeckij space $\dot W_m^\alpha $, $m\ge \max(d,2) $, $\alpha\in (0,1) $ and an additive Brownian noise). 
 \end{remark}
\begin{remark}[About other driving noises] Another natural extension would concern the class of noises considered.
One might wonder e.g. if we could consider dynamics of the form
\begin{equation}\label{STAB_DRIVEN_SDE}
 dX_t=b(t,X_t)dt+dZ_t,
 \end{equation}
with $Z $ being a rotationally invariant stable process of index $\gamma\in (1,2) $. We think that the techniques used to prove Theorem \ref{THM_HK_CONT} below, adapting somehow the strategy of \cite{meno:pesc:zhan:21}, should be sufficiently robust to obtain an error estimation with order $
h^{\frac{{\alpha}} \gamma}$ for ${\alpha=1-\frac{d}{(\gamma-1)\rho}-\frac \gamma {(\gamma-1)q}}$ when $b\in L^q-L^\rho$ with $\frac{d}{(\gamma-1)\rho}+\frac \gamma {(\gamma-1)q}<1$. Let us mention that the well-posedness in a strong sense and in a weak sense of \eqref{STAB_DRIVEN_SDE} with singular drift coefficient has respectively been addressed in \cite{zhan:13} and \cite{chau:meno:19} and that heat-kernel bounds for multiplicative stable noise and unbounded drift  have been obtained in \cite{meno:zhan:20} even in the super-critical case $\gamma\in (0,1) $ through the approach of \cite{meno:pesc:zhan:21}. The indicated thresholds can be derived following the procedure below adapting Lemma \ref{GAUSS_INT_LEMMA} to the pure jump stable case.
\end{remark}



   The article is organized as follows. We will first prove our main result in Section \ref{SEC_M_PROOF}.
   To this end, we also state therein two propositions giving Gaussian heat kernel estimates and Duhamel representations for the transition densities of the Euler scheme and the SDE \eqref{eds}. Section \ref{SEC_EUL_CTR} is dedicated to the proof of the estimates for the approximation scheme. In section \ref{Sec-Diff}, we deduce the estimates for the diffusion by letting the time-step $h\to 0$. 
Technical results are gathered in the Appendix. 


   We denote from now on by $C$ a generic constant that may change from line to line and might depend on $b,q,\rho,d,T $. Other possible dependencies will be explicitly specified. We reserve the notation $c>1$ for the concentration constant, or \textit{variance}, in the Gaussian kernels $g_c$.  For a multi-index $\zeta\in \N^d,\ x\in \R^d$, we denote  $\nabla_x^\zeta:=\partial_{x_1}^{\zeta_1}\cdots \partial_{x_d}^{\zeta_d}$. If $|\zeta|:=\sum_{i=1}^d\zeta_i =0 $, $\nabla_x^\zeta $ then simply means that there is no differentiation. Also,  for $a,b>0$, $B(a,b)=\int_0^1u^{a-1}(1-u)^{b-1}du$ stands for the $\beta $-function.

Eventually, we will  consider from now on that the condition \eqref{COND_KR} is met.

\section{Proof of the convergence rate for the error}\label{SEC_M_PROOF}
We prove in this section our main result, Theorem \ref{MTHM}. To this end, we first give two auxiliary results about density/heat kernel estimates for both the scheme and the diffusion.
 
\subsection{Key results for the proof of  Theorem \ref{MTHM}}


Using an approach inspired from \cite{meno:pesc:zhan:21}, we obtain the following estimations for the scheme.
\begin{prop}[Density estimates for the Euler scheme]\label{EST_DENS_SCHEME}
 Assume \eqref{COND_KR}. Set $h=\frac{T}{n},\ n\in \mathbb N^*$. Then the Euler scheme $X^h$ with dynamics \eqref{eulerdiff} 
 (resp. $\bar X^h$ with dynamics \eqref{euler_BIS})
 admits for all $0\le t_k:=kh<t\le T ,\ k\in  \leftB 0,n\rightB, (x,y)\in (\R^d)^2$ a transition density $\Gamma^h(t_k,x,t,y) $ (resp. $\bar\Gamma^h(t_k,x,t,y)$) which enjoys the following Duhamel representation :
 \begin{align}
 \Gamma^h(t_k,x,t,y)
  =g_1(t-t_k,y-x)\textcolor{blue}{-}\int_{t_k}^{ t}\E\left[b_h(U_{\lfloor \frac rh\rfloor },X^h_{\tau_r^h})\cdot\nabla_y g_1(t-r,y-X^h_r)\right]dr,\label{DUHAMEL_SCHEME_PROP}\\
  \mbox{resp. } \bar\Gamma^h(t_k,x,t,y)
  =g_1(t-t_k,y-x)\textcolor{blue}{-}\int_{t_k}^{ t}\E\left[\bar b_h(U_{\lfloor \frac rh\rfloor },\bar X^h_{\tau_r^h})\cdot\nabla_y g_1(t-r,y-\bar X^h_r)\right]dr.\end{align} 
 Furthermore, for each $c>1$, there exists a finite constant $C$ not depending on $h=\frac{T}{n}$ such that for all $k\in\leftB 0,n-1\rightB,
  t\in (t_k,T], x,y,y'\in\R^d$,
\begin{align}
  &\Gamma^h(t_k,x,t,y)\le C g_c(t-t_k,y-x)\label{estigausschem}
\end{align}
and  if $\alpha<1$,
\begin{align}
 \hspace*{4.2cm}   &
    |\Gamma^h(t_k,x,t,y')-\Gamma^h(t_k,x,t,y)|\notag\\
    \le& 
    C \frac{|y-y'|^{\alpha}\wedge (t-t_k)^{\frac{\alpha} 2}}{(t-t_k)^{\frac{\alpha} 2}}\big(g_c(t-t_k,y-x)+g_c(t-t_k,y'-x)\big).\label{estischemspace}
  \end{align}
Also, for all $0\le k<\ell<n,\; t\in[t_\ell,t_{\ell+1}], x,y\in\R^d$,
  \begin{align}
\hspace*{3.4cm}  &|\Gamma^h(t_k,x,t,y)-\Gamma^h(t_k,x,t_\ell,y)|\notag\\
  \le& C \frac{(t-t_\ell)^{\frac{\alpha} 2}}{(t_\ell-t_k)^{\frac {\alpha} 2}}\left(1+\I_{\{\alpha=1\}}\ln\left(\frac{t_\ell-t_k}{h}\right)\right)g_c(t-t_k,y-x)\label{estischemtemps},
\end{align}
and the same estimations hold with $\bar\Gamma^h$ replacing $\Gamma^h$. 
\end{prop}


\begin{remark}Suppose that $b\in L^\infty-L^\infty$ which corresponds to $\alpha=1$. Then $b$ also belongs to $L^q-L^\infty$ for each $q\in(2,+\infty)$ and the corresponding cutoff does not play any role if $B\ge \|b\|_{L^\infty-L^\infty}T^{\frac{1}{q}}$. As a consequence, \eqref{estischemspace} holds with $\alpha$ replaced by $\check \alpha\in (0,1)$ in the right-hand side with a multiplicative constant $C$ possibly depending on $\check \alpha$. The same is of course true for $\bar\Gamma^h$.
\end{remark}
In the limit $h\to 0$, we will deduce the following proposition.

\begin{prop}\label{THM_HK_CONT}Assume \eqref{COND_KR}. Let $(X_t)_{t\in [0,T]}$ denote the solution to the SDE \eqref{sdediff}.
  Then for each $t\in (0,T]$, $X_t$ admits a density $y\rightarrow \Gamma(0,x,t,y) $ with respect to the Lebesgue measure such that for each $c>1$, there exists $C<+\infty$ such that for all $t\in (0,T],\; x,y,y'\in \R^d $,
\begin{align}
\label{HK_AND_GRAD_SPACE_DENS}
|\Gamma(0,x,t,y)|\le C{ g}_c(t,y-x)
\end{align}
 and 
if $\alpha<1$,
\begin{align}
\;|\Gamma(0,x,t,y)-\Gamma(0,x,t,y')|
\le C \frac{|y-y'|^{\alpha}\wedge(t-s)^{\frac{\alpha}2}}{(t-s)^{\frac{\alpha}2}}
                                                                                                                         \Big({ g}_c(t,y-x)+{ g}_c(t,y'-x)\Big)\label{HKholdspace}.\end{align}
                                                                                                                         
This density enjoys the following Duhamel representation : for all $t\in (0,T]$, $(x,y)\in (\R^d)^2 $:
\begin{align}
 \Gamma(0,x,t,y)
  =g_1(t,y-x)\textcolor{blue}{-}\int_{0}^{ t}\E\left[b(r,X_{r})\cdot\nabla_y g_1(t-r,y-X_r)\right]dr.\label{DUHAMEL_DIFF_PROP}
  \end{align}
\end{prop}



Let us mention that similar Gaussian estimates were obtained for even rougher drifts in Besov spaces with negative regularity index by Perkowski and Van Zuijlen in \cite{perk:vanz:20} using Littlewood-Paley decompositions for the drift. For time-homogeneous drifts, heat kernel estimates of the same type were obtained by Zhang and Zhao \cite{zhan:zhao:17}, see Theorem 5.1 therein, under the condition $\|(I-\Delta)^{-\frac \alpha 2}b\|_{\rho}<+\infty,\ \alpha\in (0,\frac 12), \ \rho\in (\frac{d}{1-\alpha},+\infty) $ through change of probability techniques. 
However, it seems that the $L^q-L^\rho$ case had not been considered so far. Hence, to the best of our knowledge, the above heat-kernel estimates are new and can be of interest independently of the approximation procedure.

\subsection{Proof of the main Theorem for the discretization error}\label{SEC_PROOF_MAIN_RESULT}
From the results of the previous subsection we are almost in position to prove our main result: the error bound of Theorem \eqref{MTHM} for the densities. We state in the next paragraph the additional technical lemmas also needed.  
\subsubsection{Preliminary results}
We state here three technical lemmas that turn out to be useful for the error analysis. The first one is related to integrability properties of Gaussian kernels integrating a function in $L^q-L^\rho $. Such kind of integrals appear from the decomposition of the error \eqref{THIRD_EXP_ERR}. The second one gives standard quantitative bounds for Gaussian kernels. The last one is a Gronwall-Volterra lemma.
Before stating our \textit{Gaussian lemmas}, we recall that for $c,u>0$, $g_c(u,\cdot)$ denotes the centered Gaussian density with covariance matrix $cuI_d$.
\begin{lem}[Singularities induced by an $L^{q'}-L^{\rho'} $ drift in a Gaussian convolution]\label{GAUSS_INT_LEMMA}
Let $\rho',q'\in[1,+\infty]$. It holds that there exists $C:=C(\rho',q',d)$ s.t. for all $0\le s<t\le T $, all $\beta,\gamma\ge 0$,  and all measurable functions $\varphi:[0,t]\times\R^d\to\R$ and $f:[s,t]\rightarrow \R_+$.
\begin{align*}
&I_{\beta,\gamma,f,\varphi}(s,t)\\
:=&\int_{s}^t \frac{du f(u)}{(u-s)^\beta(t-u)^{\gamma}}\int_{\R^d} g_c(u-s,z-x) |\varphi(u,z)| g_c(t-u,y-z)dz\\
\le& C (t-s)^{\frac d{2\rho'}} g_c(t-s,y-x) \Bigg(\int_s^t du\left(  \frac{f(u)}{(u-s)^{\beta+\frac{d}{2\rho'}}(t-u)^{\gamma+\frac{d}{2\rho'}} } \right)^{\bar q'}\Bigg)^{\frac 1{\bar q'}}\|\varphi\|_{L^{q'}-L^{\rho'}},
\end{align*}
and $\frac{1}{q'}+\frac{1}{\bar q'}=1$.
When $f$ is bounded, or in particular when $f$ is constant, the time singularities are integrable provided that $(\beta+\frac{d}{2\rho'})\vee (\gamma+\frac{d}{2\rho'}) <\frac 1{\bar q'}=1-\frac{1}{q'}$. In that case,
\begin{align}
\label{CTR_SING_WITH_CONSTANT}
I_{\beta,\gamma,f,\varphi}(s,t)\le C\|f\|_{L^\infty}\|\varphi\|_{L^{q'}-L^{\rho'}} g_c(t-s,y-x)(t-s)^{1-\frac 1{q'}-(\beta+\gamma+\frac{d}{2\rho'})}\notag\\
\times \left(B\left(1-\bar q'(\beta+\frac{d}{2\rho'}),1-\bar q'(\gamma+\frac{d}{2\rho'})\right)\right)^{\frac 1{\bar q'}}.
\end{align}
\end{lem}
\begin{proof}
Set $G_{s,t}^{x,y}(u):=\int_{\R^d} g_c(u-s,z-x) |\varphi(u,z)| g_c(t-u,y-z)dz $. From the H\"older inequality, we get that there exists a finite constant $C_{\rho'}$ s.t.
\begin{align}
|G_{s,t}^{x,y}(u)|
\le& C_{\rho'}\|\varphi(u,\cdot)\|_{L^{\rho'}}\left( \int_{\R^d} \frac{g_{\frac c{\bar \rho'}}(u-s,x-z)g_{\frac c{\bar \rho'}}(t-u,y-z)}{(u-s)^{(\bar \rho'-1) \frac d2}(t-u)^{(\bar \rho'-1) \frac d2}}dz\right)^{\frac {1}{\bar \rho'}},\ \frac {1}{\rho'}+\frac 1{\bar \rho'}=1\notag\\
\le& C_{\rho'}\frac{\|\varphi(u,\cdot)\|_{L^{\rho'}}}{(u-s)^{\frac{d}{2\rho'}}(t-u)^{\frac{d}{2\rho'}}}\times\frac{1}{(t-s)^{\frac{d}{2\bar \rho'}}}\exp\Big(-\frac{|x-y|^2}{2c(t-s)} \Big)\notag\\
\le &C_{\rho'}\frac{\|\varphi(u,\cdot)\|_{L^{\rho'}}(t-s)^{\frac{d}{2\rho'}}}{(u-s)^{\frac{d}{2\rho'}}(t-u)^{\frac{d}{2\rho'}} } g_c(t-s,x-y), \label{BD_G_EPS}
\end{align}
up to a modification of $C_{\rho'}$ from line to line. From the definition of $ I_{\beta,\gamma,f,\varphi}(s,t)$ we derive the statement from \eqref{BD_G_EPS} and the H\"older inequality (in time). Namely,
\begin{eqnarray*}
&&I_{\beta,\gamma,f,\varphi}(s,t)\\
&\le& C_{\rho'}(t-s)^{\frac{d}{2\rho'}}g_c(t-s,x-y)\int_{s}^t du f(u) \frac{\|\varphi(u,\cdot)\|_{L^{\rho'}}}{(u-s)^{\beta+\frac{d}{2\rho'}}(t-u)^{\gamma+\frac{d}{2\rho'}} } \\
&\le & C (t-s)^{\frac d{2\rho'}} g_c(t-s,y-x) \Bigg(\int_s^t du\left(  \frac{f(u)}{(u-s)^{\beta+\frac{d}{2\rho'}}(t-u)^{\gamma+\frac{d}{2\rho'}} } \right)^{\bar q'}\Bigg)^{\frac 1{\bar q'}}\|\varphi\|_{L^{q'}-L^{\rho'}}.
\end{eqnarray*}
The integrability conditions and the explicit control of \eqref{CTR_SING_WITH_CONSTANT} then readily follow when $f$ is bounded. The proof is complete.
\end{proof}

For the computations to be performed, we will also often need quantitative bounds for sensitivities of Gaussian kernels.
We state the following result the proof of which is \textit{standard} and postponed to Appendix \ref{SEC_PROOF_USUAL_GAUSS_EST} for the sake of completeness.
\begin{lem}[Gaussian Sensitivities]\label{LEM_GAUS_SENS}
For each $c>1$ there exists $C<+\infty$ s.t. for each  multi-index $\zeta$ with length $|\zeta|\le  2 $, and for all $0< u\le u' \le T $, $x,x'\in \R^d $ :
\begin{equation}
\label{CTR_GRAD_GAUSS}
|\nabla_x^\zeta {\tildeGamma}(u,x)|\le \frac{C}{u^{\frac{|\zeta|}2}}g_c(u,x)\mbox{ and } |\partial_u\nabla_x^\zeta {\tildeGamma}(u,x)|\le \frac{C}{u^{1+\frac{|\zeta|}2}}g_c(u,x),
\end{equation}
\begin{align}\label{DIFF_GRAD_GAUSS_LEM}
\Big|\nabla^\zeta _x {\tildeGamma}(u,x)-\nabla^\zeta _x {\tildeGamma}(u,x')\Big|\le C\frac{|x-x'|\wedge u^{\frac{1}{2}}}{u^{\frac{1+|\zeta|}2}}(g_c(u,x)+g_c(u,x')),
\end{align}
\begin{align}\label{DIFF_GRAD_GAUSS_LEM_TIME}
  \Big|\nabla_x^\zeta {\tildeGamma}(u',x)-\nabla_x^\zeta {\tildeGamma}(u,x)\Big|\le C\frac{|u'-u|\wedge u}{u^{1+\frac{|\zeta|}2}}(g_c(u,x)+g_c(u',x)).
\end{align}
\end{lem}

The next lemma, the proof of which is postponed to  Appendix \ref{APP_GR} roughly says that the usual Gronwall inequality extends to integral inequalities involving integrable singularities. 
\begin{lem}[Gronwall-Volterra Lemma]\label{GR_VOL_LEMMA}
 \begin{description}
   \item [(i)] Let $\tilde \beta<1$, $\beta>\tilde\beta-1$ and $\eta,\delta,T>0$. There exists some finite constant $C_{\beta,\tilde \beta,\eta,\delta,T}$ such that $\sup_{t\in[0,T]} f(t)\le C_{\beta,\tilde \beta,\eta,\delta,T}$ for each measurable and bounded function $f:[0,T]\to\R_+$ satisfying
  \begin{equation}
   \forall t\in[0,T],\,f(t)\le \eta+\delta t^{\beta}\int_0^t\frac{f(s)ds}{s^{\tilde \beta}}.\label{sansdeux}
 \end{equation}
\item[(ii)]Let $\tilde \beta,\hat\beta<1$, $\check \beta>\tilde \beta+\hat\beta-1$ and $a,b,T>0$. There exists some finite constant $C_{\tilde \beta,\hat\beta,\check \beta,a,b,T}$ such that $\sup_{t\in[0,T]} f(t)\le C_{\tilde \beta,\hat\beta,\check \beta,a,b,T}$ for each measurable and bounded function $f:[0,T]\to\R_+$ satisfying
  \begin{equation}
   \forall t\in[0,T],\,f(t)\le a+b t^{\check \beta}\int_0^t\frac{f(s)ds}{s^{\tilde \beta}(t-s)^{\hat\beta}}.\label{avecdeux}
  \end{equation}
  \end{description}
  
\end{lem}
\begin{remark}
  Under the assumptions of {\bf (i)}, $\forall t\in[0,T]$, $f(t)\le \eta+\frac{\delta t^{\beta+1-\tilde\beta}}{1-\tilde\beta}\sup_{s\in[0,t]}f(s)$ so that $\sup_{s\in[0,t]}f(s)\le \eta+\frac{\delta t^{\beta+1-\tilde\beta}}{1-\tilde\beta}\sup_{s\in[0,t]}f(s)$ and when $t<\left(\frac{1-\tilde\beta}{\delta}\right)^{\frac{1}{\beta+1-\tilde\beta}}$, $$\sup_{s\in[0,t]}f(s)\le \frac{\eta(1-\tilde\beta)}{1-\tilde\beta-\delta t^{\beta+1-\tilde\beta}}.$$\\Similarly, under the assumptions of {\bf (ii)}, for $t\in[0,T]$ such that $t<(bB( 1-\tilde\beta,1-\hat\beta))^{-\frac{1}{\check\beta+1-\tilde\beta-\hat\beta}}$, $
\sup_{s\in[0,t]}f(s)\le \frac{a}{1-bB(1-\tilde\beta,1-\hat\beta)t^{\check\beta+1-\tilde\beta-\hat\beta}}$.\end{remark}

\subsubsection{Final derivation of the error bounds}\label{THE_SEC_FOR_PROOF_MTHM}
By \eqref{DUHAMEL_SCHEME_PROP} and \eqref{DUHAMEL_DIFF_PROP}, the discretization error writes
\begin{align*}
  &\Gamma^h(0,x,t,y)-\Gamma(0,x,t,y)\\
  =&\E\left[\int_0^t \Big( b(s, X_{s})\cdot \nabla_y g_1(t-s,y-X_s)-b_h(U_{\lfloor \frac sh\rfloor}, X_{\tau_s^h}^h)\cdot \nabla_y g_1(t-s,y-X_s^h)\Big)ds\right].\end{align*}
For $s\in(t_1,T]\setminus\{t_k:k\in\llbracket2,n-1\rrbracket\}$, $\varphi:\R^d\times\R^d\times\R\to\R$ measurable and bounded, we have using  $X^h_s=X^h_{\tau^h_s}+W_s-W_{\tau^h_s}+b_h(U_{\lfloor \frac s h\rfloor},X^h_{\tau^h_s})(s-\tau^h_s)$ and the independence of $X^h_{\tau^h_s}$, $W_s-W_{\tau^h_s}$ and $U_{\lfloor \frac s h\rfloor}$,
\begin{align*}
 &  \E\left[\varphi(X^h_{\tau^h_s},X^h_s,U_{\lfloor \frac s h\rfloor})\right]\\=&\frac 1 h\int_{\tau^h_s}^{\tau^h_s+h}\int_{\R^d\times\R^d}\varphi(w,z,r)\Gamma^h(0,x,\tau^h_s,w)g_1\left(s-\tau^h_s,z-w-b_h(r,w)(s-\tau^h_s)\right)dzdwdr.
\end{align*}

We deduce that the error decomposes as

\begin{align}
 &\Gamma^h(0,x,t,y)-\Gamma(0,x,t,y)=\Delta^1_t+\Delta^2_t+\Delta^3_t+\Delta^4_t+\Delta^5_t+\Delta^6_t\mbox{ where }\label{decomperr}\\
 &\Delta^1_t=\int_0^t ds  \int_{\R^d}[\Gamma(0,x,s,z)-\Gamma^h(0,x,s,z)] b(s,z) \cdot \nabla_y g_1(t-s,y-z)dz,\notag\\
  &\Delta^2_t=\I_{\{t\ge 3h\}}\int_{t_1}^{\tau_t^{{h}}-h}ds\int_{\R^d}\Gamma^h(0,x,s,z)(b(s,z)-b_h(s,z)).\nabla_y g_1(t-s,y-z)dz,\notag\\&\Delta^3_t=\I_{\{t\ge 3h\}}\int_{t_1}^{\tau_t^{{h}}-h} ds\int_{\R^d}[\Gamma^h(0,x,s,z)-\Gamma^h(0,x,\tau^h_s,z)]b_h(s,z).\nabla_y g_1(t-s,y-z) dz,\notag
  \\ &\Delta^4_t=\I_{\{t\ge 3h\}}\int_{t_1}^{\tau_t^{{h}}-h}ds\E\bigg[b_h(U_{\lfloor s/h\rfloor},X_{\tau^h_s}^h)\cdot (\nabla_y g_1(t-U_{\lfloor s/h\rfloor},y-X_{\tau^h_s}^h)\notag\\
  &\hspace*{2cm}-\nabla_y g_1(t-s,y-X_s^h))\bigg],\notag\\
                &\Delta^5_t=\frac 1 h\int_{s=0}^{t_1\wedge t}\int_{r=0}^{h}\int_{\R^d}g_1(s,z\!-\!x\!-\!b_h(r,x)s)\left(b(s,z)\!-\!b_h(r,x)\right).\nabla_y g_1(t\!-\!s,y\!-\!z)dz drds,\notag\\&\Delta^6_t=\I_{\{t\ge h\}}\frac 1 h\int_{s=(\tau_t^{{h}}-h)\vee t_1}^{t}\int_{r=\tau^h_s}^{\tau^h_s+h}\int_{\R^{2d}}\Gamma^h(0,x,\tau^h_s,w)g_1\left(s\!-\!\tau^h_s,\!z\!-\!w\!-\!b_h(r,w)(s\!-\!\tau^h_s)\right)\notag\\&\phantom{\Delta^5_t=\I_{\{t\ge h\}}\frac 1 h\int_{s=(\tau_t^{{h}}-h)\vee t_1}^{t}}\times\left(b(s,z)-b_h(r,w)\right).\nabla_y g_1(t-s,y-z)dz dwdrds
     \notag.
\end{align}

Let us first deal with the cutoff error term $\Delta^2_t$ when $t\ge 3h$ (the contribution $\Delta^1_t$ will be handled at the end of the proof by a Gronwall type argument). When $\rho=q=\infty$, we recall that we choose $B=\|b\|_{L^\rho-L^q}$ so that $b_h=b$ and $\Delta^2_t=0$. Let us then suppose that $\frac 2q+\frac d{\rho}>0$, choose $\tilde \alpha\in[\alpha,\alpha+\frac 1 2)$ and set $\check \alpha=\frac{\tilde \alpha}{\frac 2q+\frac d{\rho}}$. We have $$|b-b_h|=\left(|b|-B h^{-(\frac 1 q+\frac{d}{2\rho})}\right)^+\le |b|\I_{\{|b|\ge Bh^{-(\frac 1 q+\frac{d}{2\rho})}\}}\le \frac{h^{\frac{\tilde \alpha} 2}|b|^{1+\check \alpha}}{B^{\check \alpha}}. $$Combining this inequality with $\tau^h_t-h\le t-h$, \eqref{CTR_GRAD_GAUSS} and \eqref{estigausschem} then applying H\"older's inequality in space like in Lemma \ref{GAUSS_INT_LEMMA} with $\rho'=\frac{\rho}{1+\check \alpha}$ and last H\"older's inequality in time, 
we obtain
\begin{align*}
  &|\Delta^2_t|\\
  \le &Ch^{\frac{\tilde \alpha} 2}\int_{0}^{t-h} ds\int_{\R^d}g_c(s,z-x)|b(s,z)|^{1+\check \alpha}\frac{g_c(t-s,y-z)}{\sqrt{t-s}}dzds\\\le &Ch^{\frac{\tilde \alpha} 2}g_c(t,y-x)t^{\frac{d(1+\alpha)}{2\rho}}\int_0^{t-h}\frac{\|b(s,.)\|_{L^\rho}^{1+\check \alpha}ds}{s^{\frac{d(1+\alpha)}{2\rho}}(t-s)^{\frac 1 2+\frac{d(1+\alpha)}{2\rho}} }\\\le &Ch^{\frac{\tilde \alpha} 2}g_c(t,y-x)\left(t^{-\frac 1 2}\int_0^{\frac t 2}\frac{\|b(s,.)\|_{L^\rho}^{1+\check \alpha}ds}{s^{\frac{d(1+\alpha)}{2\rho}}}+\int_{\frac t 2}^{t-h}\frac{\|b(s,.)\|_{L^\rho}^{1+\check\alpha}ds}{(t-s)^{\frac 1 2+\frac{d(1+\alpha)}{2\rho}} }\right)\\
  \le &Ch^{\frac{\tilde \alpha} 2}g_c(t,y-x)\|b\|_{L^q-L^\rho}^{1+\check \alpha}\left(t^{-\frac 1 2}\times t^{\frac{1+\alpha-\tilde\alpha}{2}}+\I_{\{\tilde\alpha>\alpha\}}h^{\frac{\alpha-\tilde\alpha}{2}}+\I_{\{\tilde\alpha=\alpha\}}\left(\ln\left(\frac t{2h}\right)\right)^{1-\frac{1+\check \alpha} q}\right).
\end{align*}
For the application of the H\"older inequality in space (resp. in time), we needed that $\frac{\rho}{1+\check\alpha}\ge 1$ (resp. $\frac{q}{1+\check\alpha}>1$) an inequality of course satisfied when $\rho=\infty$ (resp. $q=\infty$) and otherwise equivalent to $\frac{\rho\left(\frac{2}{q}+\frac{d}{\rho}\right)}{\frac{2}{q}+\frac{d}{\rho}+\tilde\alpha}\ge 1$ (resp. $\frac{q\left(\frac{2}{q}+\frac{d}{\rho}\right)}{\frac{2}{q}+\frac{d}{\rho}+\tilde\alpha}> 1$) and thus to $\tilde\alpha\le \alpha+d-1+\frac{2\rho} q$ (resp. $\tilde\alpha< \alpha+1+\frac{dq}{\rho}$ which always holds true). When $d\ge 2$ or $q<\infty$, we may choose $\tilde\alpha\in (\alpha,\alpha+\frac 1 2)$ so that the first requirement is satisfied as well, while when $d=1$ and $q=\infty$,  $\tilde\alpha=\alpha$ is the only possible choice. With the inequality $t\ge 3h$, we conclude that
\begin{equation}
   |\Delta^2_t|\le Ch^{\frac{\alpha} 2}\left(1+\I_{\{d=1,q=\infty,\rho<\infty\}}\ln\left(\frac Th\right)\right)g_c(t,y-x)\label{THE_ALTERNATIVE_DELTA_5_FOR_ERR}.
\end{equation}
For the scheme \eqref{euler_BIS}, note that the first time step when $\bar b_h=0$ does not contribute to the term $\bar\Delta^2_t$ analogous to $\Delta^2_t$ in the error decomposition \eqref{decomperr}. When $\rho=q=\infty$, then the cutoff error vanishes as soon has $h\le\frac{B^2}{\|b\|_{L^\infty-L^\infty}^2}$. When either $\rho<\infty$ or $q<\infty$, we may reproduce the above reasoning with $\check\alpha=\tilde\alpha$ since $|b-\bar b_h|\le \frac{h^{\frac{\tilde \alpha} 2}|b|^{1+\tilde \alpha}}{B^{\tilde \alpha}}$. The requirement for the  H\"older inequality in space writes $\tilde\alpha\le \rho-1$ where, since $\rho\ge 2$ by \eqref{COND_KR}, $\rho-1>\alpha$. Hence the logarithmic term may be removed when $d=1$, $q=\infty$ and $\rho<\infty$ :
\begin{equation*}
   |\bar\Delta^2_t|\le Ch^{\frac{\alpha} 2}g_c(t,y-x).
\end{equation*}

Concerning the estimations of $\Delta^i_t$, $i\in\{3,4,6\}$, the choice of the cutoff does not play any role (we will only use that $|b_h|\le |b|$). That is why we deal before with the estimation of the contribution $\Delta^5_t$ of the first time step where the arguments depend on this choice. The inequality
\begin{equation}
   \forall c'>1,\;\forall x,y,z\in\R^d,\;\left|z-x-y\right|^2\ge \frac{1}{c'}|z-x|^2-\frac{1}{c'-1}|y|^2\label{majodiff}
\end{equation}
applied with $c'=c$ and $y=sb_h(r,x)$ such that $|y|\le \frac{Bs}{h^{\frac 1q+\frac d{2\rho}}}$ by the definition \eqref{cutoffb} of  $b_h$, implies that,
\begin{equation}
 \forall (s,r,x,z)\in(0,h]\times[0,T]\times\R^d\times\R^d,\;g_1(s,z-x-b_h(r,x)s)\le c^{\frac{d}{2}}e^{\frac{B^2h^{\alpha}}{2(c-1)}}g_c(s,z-x).\label{estig}
\end{equation}
When $b_h$ is replaced by $\bar b_h$, the factor $c^{\frac{d}{2}}e^{\frac{B^2h^{\alpha}}{2(c-1)}}$ should be replaced by $c^{\frac{d}{2}}e^{\frac{B^2}{2(c-1)}}$. 

Equation \eqref{estig} is crucial since it precisely emphasizes that for small  time transitions of the scheme the chosen cutoffed drift is negligible with respect to the diffusive behavior of the Brownian motion.

With \eqref{CTR_GRAD_GAUSS} and  the definition \eqref{cutoffb} of  $b_h$, then H\"older's inequality in space, that $t-s\ge \frac{t}{2}$ for $s\in[0,t_1]$ when $t\ge 2h$, H\"older's inequality in time, we obtain that
\begin{align}
  |\Delta^5_t|&\le \frac{C}{h}\int_{s=0}^{t_1\wedge t}\int_{r=0}^h\int_{\R^d}g_c(s,z-x)\left(|b(s,z)|+|b(r,x)|\wedge Bh^{-(\frac 1q+\frac d{2\rho})}\right)\notag\\
  &\hspace*{2cm}\times \frac{g_c(t-s,y-z)}{\sqrt{t-s}}dzdrds\notag\\
            &\le Cg_c(t,y-x)\left(t^{\frac d{2\rho}}\int_0^{t_1\wedge t}\frac{\|b(s,.)\|_{L^\rho}ds}{s^{\frac{d}{2\rho}}(t-s)^{\frac 1 2+\frac{d}{2\rho}}}+B h^{-(\frac 1q+\frac d{2\rho})}\int_0^{t_1\wedge t}\frac{ds}{\sqrt{t-s}}\right)\notag\\&\le Cg_c(t,y-x)\bigg(\I_{\{t<2h\}}\left(t^{\frac d{2\rho}}\int_0^{t}\frac{\|b(s,.)\|_{L^\rho}ds}{s^{\frac{d}{2\rho}}(t-s)^{\frac 1 2+\frac{d}{2\rho}}}+B h^{-(\frac 1q+\frac d{2\rho})}\int_0^{t}\frac{ds}{\sqrt{t-s}}\right)\notag\\&\phantom{\le Cg_c(t,y-x)\bigg(}+\I_{\{t\ge 2h\}}t^{-\frac 12}\left(\int_0^{t_1}\frac{\|b(s,.)\|_{L^\rho}ds}{s^{\frac{d}{2\rho}}}+Bh^{-(\frac 1 q+\frac d{2\rho})}t_1\right)\bigg)\notag\\
 &\le Cg_c(t,y-x)(\|b\|_{L^q-L^\rho}+B)\left(\I_{\{t<2h\}}[t^{\frac{\alpha} 2}+h^{\frac{\alpha} 2}]+\I_{\{t\ge 2h\}}h^{\frac{\alpha}{2}}\right)\notag\\ &\le Cg_c(t,y-x)h^{\frac{{\alpha}} 2}\label{THE_ALTERNATIVE_DELTA_1_FOR_ERR} .
\end{align}
For the scheme \eqref{euler_BIS}, we have removed the drift on the first time step to get rid of the contribution of $|b(r,x)|\wedge (Bh^{-\frac 12})$ in the previous analysis which would have led to a bound in $(t\wedge h)^{\frac 1 2}(t\vee h)^{-\frac 12}g_c(t,y-x)$. 

We next suppose that $t\ge 3h$ to estimate the error contribution $\Delta^3_t$, since this contribution vanishes otherwise. Using \eqref{estischemtemps}, \eqref{CTR_GRAD_GAUSS} and $|b_h|\le |b|$, $\tau^h_s\ge \frac s2$ when $s\ge t_1$, then applying Lemma \ref{GAUSS_INT_LEMMA} with $\rho'=\rho$, $q'=q$, $\varphi=|b|$, $f=1$, $\beta=\frac{{\alpha} }2$, $\gamma=\frac 12$, we obtain that
\newpage
\begin{align}
   &|\Delta^3_t|\notag\\
   \le &C\left(1+\I_{\{\alpha=1\}}\ln\left(\frac Th\right)\right)\int_{t_1}^{\tau_t^{{h}}-h} ds \frac{(s-\tau^h_s)^{\frac {{\alpha} }2}}{(\tau^h_s)^{\frac {{\alpha}} 2}}\int_{\R^d}g_c(s,z-x)|b(s,z)|\frac{g_c(t-s,y-z)}{\sqrt{t-s}}dz\notag\\\le &C \left(1+\I_{\{\alpha=1\}}\ln\left(\frac Th\right)\right)h^{\frac{{ \alpha}} 2}\int_{0}^{t} \frac{ds}{s^{\frac{{\alpha}} 2}}\int_{\R^d}g_c(s,z-x)|b(s,z)|\frac{g_c(t-s,y-z)}{\sqrt{t-s}}dz\notag\\\le& C h^{\frac{{\alpha} }2}\left(1+\I_{\{\alpha=1\}}\ln\left(\frac Th\right)\right)
   \|b\|_{L^q-L^\rho}g_c(t,y-x).\label{THE_ALTERNATIVE_DELTA_4_FOR_ERR}
\end{align}

We still suppose that $t\ge 3h$ to estimate the error contribution $\Delta^4_t$, since this contribution vanishes otherwise. Using now $|b_h|\le |b|$ for the first inequality, \eqref{estigausschem}, \eqref{DIFF_GRAD_GAUSS_LEM_TIME}, \eqref{estig} with $(c,s,x,z)$ replaced by $(\frac{1+c}{2},s-t_j,z,w)$ and \eqref{DIFF_GRAD_GAUSS_LEM} for the second inequality, 
$$|z-w|^\alpha g_{\frac{1+c}{2}}(s-t_j,w-z)\le \left(\frac{2c}{1+c}\right)^{\frac d 2}\sup_{a\ge 0} a^\alpha e^{-\frac{(c-1)a^2}{2c(1+c)}}(s-t_j)^{\frac \alpha 2}g_c(s-t_j,w-z)$$ and $g_{\frac{1+c}{2}}\le \left(\frac{2c}{1+c}\right)^{\frac d 2}g_c$ combined with Gaussian convolution and $|r-s|\vee (s-t_j)\le h$ for the third inequality, H\"older's inequality in space for the fourth inequality, we obtain for $\tilde \alpha \in[\alpha,1]$ to be specified later on :
\begin{align*}
  &|\Delta^4_t|\\
  \le& \sum_{j=1}^{\lfloor \frac t h\rfloor -2}\frac 1 h\int_{t_j}^{t_{j+1}}\!\! ds\int_{t_j}^{t_{j+1}}\!\! dr\int_{\R^d\times\R^d}\Gamma^h(0,x,t_j,z)g_1(s-t_j,w-z-b_h(r,z)(s-t_j))|b(r,z)|\\
   &
  \times \left(|\nabla_y g_1(t\!-\!r,y\!-\!z)\!\!-\nabla_y g_1(t\!-\!s,y\!-\!z)|
  +|\nabla_y g_1(t\!-\!s,y\!-\!z)\!-\!\nabla_y g(t\!-\!s,y\!-\!w)|\right)dzdw\\
   \le & \frac C h\sum_{j=1}^{\lfloor \frac t h\rfloor -2}\int_{t_j}^{t_{j+1}}ds\int_{t_j}^{t_{j+1}}dr\int_{\R^d\times\R^d}g_c(t_j,z-x)g_{\frac{1+c}{2}}(s-t_j,w-z)|b(r,z)|\\&
   \left(\frac{|r\!-\!s|^{\frac {{\tilde\alpha}} 2}}{(t\!-\!r\vee s)^{\frac{1+{\tilde\alpha}}{2}}}g_c(t\!-\!r\vee s,y\!-\!z)+\frac{|z\!-\!w|^{{\tilde\alpha}}}{(t-s)^{\frac{1+{\tilde\alpha}}{2}}}(g_c(t\!-\!s,y\!-\!z)+g_c(t\!-\!s,y\!-\!w))\right)dzdw\\
  \le& C h^{\frac{{\tilde\alpha}} 2-1}
    \sum_{j=1}^{\lfloor \frac t h\rfloor -2}\int_{t_j}^{t_{j+1}}ds\int_{t_j}^{t_{j+1}}dr\int_{\R^d}\frac{g_c(t_j,z-x)|b(r,z)|}{(t-r\vee s)^{\frac{1+{\tilde\alpha}}{2}}}\Big(g_c(t-r\vee s,y-z)\\
    &\phantom{\le C h^{\frac\alpha 2-1}
    \sum_{j=1}^{\lfloor \frac t h\rfloor -1}\int_{t_j}^{t_{j+1}}ds\int_{t_j}^{t_{j+1}}dr\int_{\R^d}}+g_c(t-s,y-z)+g_c(t-t_j,y-z)\Big)dz
     \\\le& Ch^{\frac{{\tilde\alpha}} 2-1}\sum_{j=1}^{\lfloor \frac t h\rfloor -2}\int_{t_j}^{t_{j+1}}ds\int_{t_j}^{t_{j+1}}dr\frac{t^{\frac{d}{2\rho}}\|b(r,.)\|_{L^\rho}}{t_j^{\frac{d}{2\rho}}(t-r\vee s)^{\frac{1+{\tilde\alpha}}{2}+\frac{d}{2\rho}}}\Big(g_c(t+t_j-r\vee s,y-x)\\
     &\phantom{\le C h^{\frac\alpha 2-1}
    \sum_{j=1}^{\lfloor \frac t h\rfloor -1}\int_{t_j}^{t_{j+1}}ds\int_{t_j}^{t_{j+1}}dr\int_{\R^d}}+g_c(t+t_j-s,y-x)+g_c(t,y-x)\Big).
\end{align*}
Since $t\ge 3h$, for $r,s\in[t_j,t_{j+1}]$ with $j\le \lfloor \frac t h\rfloor -2$, $t\ge t+t_j-s\ge t+t_j-r\vee s\ge t-h\ge \frac{2t}{3}$ so that $$g_c(t+t_j-r\vee s,y-x)+g_c(t+t_j-s,y-x)\le \frac{3^{\frac d 2}}{ 2^{\frac d 2-1}}g_c(t,y-x).$$
Since for $r,s\le \tau_t^{{h}}-h$ in the same time-step, $(t-r\vee s)\ge \frac{t-r}{2}$ and $t_j\ge \frac{r}{2}$ for $r\in[t_j,t_{j+1}]$ when $j\ge 1$  
, we deduce that
\begin{align*}
  &\frac{|\Delta^4_t|}{g_c(t,y-x)}\le Ch^{\frac{{\tilde\alpha}} 2}t^{\frac{d}{2\rho}}\int_0^{\tau_t^{{h}}{-h}}\frac{\|b(r,.)\|_{L^\rho}dr}{r^{\frac{d}{2\rho}}(t-r)^{\frac{1+{\tilde\alpha}}{2}+\frac{d}{2\rho}}}\\
  &\le Ch^{\frac{{\tilde\alpha}} 2}\left(t^{-\frac{1+{\tilde\alpha}}{2}}\int_0^{\frac{t}{2}}\frac{\|b(r,.)\|_{L^\rho}dr}{r^{\frac{d}{2\rho}}}+\int_{\frac{t}{2}}^{t-h}\frac{\|b(r,.)\|_{L^\rho}dr}{(t-r)^{\frac{1+{\tilde\alpha}}{2}+\frac{d}{2\rho}}}\right)\\&\le Ch^{\frac{{\tilde\alpha}} 2}\left( t^{-\frac{1+{\tilde\alpha}}{2}}\|b\|_{L^q-L^\rho}t^{\frac{1+{\alpha}}{2}}+\|b\|_{L^q-L^\rho}\left(\I_{\{\tilde\alpha>\alpha\}}h^{\frac{\alpha-\tilde\alpha}{2}}+\I_{\{\tilde\alpha=\alpha\}}\left(\ln\left(\frac t{2h}\right)\right)^{1-\frac 1 q}\right)\right)\label{THE_ALTERNATIVE_DELTA_3_FOR_ERR},\end{align*}
where we used H\"older's inequality and $\frac{q}{q-1}\left(\frac{1+{\tilde\alpha}}{2}+\frac{d}{2\rho}\right)\ge \frac{q}{q-1}\left(\frac{1+{\alpha}}{2}+\frac{d}{2\rho}\right)= 1$ since $\tilde \alpha\ge \alpha$ for the last inequality. When $\alpha<1$ (i.e. either $\rho<\infty$ or $q<\infty$), we choose $\tilde\alpha\in (\alpha,1]$, while when $\alpha=1$ (i.e. $\rho=q=\infty$), the only possible choice is $\tilde\alpha=\alpha=1$.  Using that $t\ge 3h$, we conclude that
\begin{equation}
   |\Delta^4_t|\le C h^{\frac{\alpha}{2}}\left(1+\I_{\{\alpha=1\}}\ln\left(\frac Th\right)\right)g_c(t,y-x).\label{THE_ALTERNATIVE_DELTA_3_FOR_ERR}
\end{equation}

Let us now suppose that $t\ge h$ to estimate $\Delta^6_t$. Using \eqref{estig} with $(s,x)$ replaced by $(s-\tau^h_s,w)$, \eqref{CTR_GRAD_GAUSS} and $|b_h|\le |b|$ then Gaussian convolution, H\"older's inequality in space, last that $s\ge \tau^h_s\ge \frac t 3$ for $s\ge (\tau_t^{{h}}-h)\vee t_1$ with $t\ge h$ ($\tau^h_s\ge h\ge\frac{t}{3}$ when $t\in[h,3h]$ while when $t>3h$, $\tau^h_t-h>\frac{t}{2}$), $t-\tau^h_s\ge t-s$  and H\"older's inequality in time, we obtain that
\begin{align}
  |\Delta^6_t|\le &\frac{C}{h}\int_{s=(\tau_t^{{h}}-h)\vee t_1}^{t}\int_{r=\tau^h_s}^{\tau^h_s+h}\int_{\R^d\times \R^d}g_c(\tau^h_s,w-x)g_c(s-\tau^h_s,z-w)\notag\\
  &\hspace*{2cm}\times\left(|b(s,z)|+|b(r,w)|\right)\frac{g_c(t-s,y-z)}{\sqrt{t-s}}dz dw drds\notag\\\le &C\int_{(\tau_t^{{h}}-h)\vee t_1}^{t}\int_{\R^d}g_c(s,z-x)|b(s,z)|\frac{g_c(t-s,y-z)}{\sqrt{t-s}} dz ds\notag\\&+\frac C h\int_{s=(\tau_t^{{h}}-h)\vee t_1}^{t}\int_{r=\tau^h_s}^{\tau^h_s+h}\int_{\R^d}|b(r,w)|g_c(\tau^h_s,w-x)\frac{g_c(t-\tau^h_s,y-w)}{\sqrt{t-s}}dw dsdr\notag\\
            \le &Cg_c(t,y-x)\left(t^{\frac d{2\rho}}\int_{(\tau_t^{{h}}-h)\vee t_1}^{t}\frac{\|b(s,.)\|_{L^\rho}ds}{s^{\frac{d}{2\rho}}(t-s)^{\frac 1 2+\frac{d}{2\rho}}}\right.\notag\\
            &\left.+\frac{t^{\frac d{2\rho}}}{h}\int_{s=(\tau_t^{{h}}-h)\vee t_1}^{t}\int_{r=\tau^h_s}^{\tau^h_s+h}\frac{\|b(r,.)\|_{L^\rho}}{(\tau^h_s)^{\frac d{2\rho}}(t-\tau^h_s)^{\frac{d}{2\rho}}}dr\frac{ds}{\sqrt{t-s}}\right)\notag\\\le &C\|b\|_{L^q-L^\rho}g_c(t,y-x)\left(\left(t-(\tau_t^{{h}}-h)\vee t_1\right)^{\frac{\alpha}{2}}+h^{-\frac 1 q}\int_{(\tau_t^{{h}}-h)\vee t_1}^{t}\frac{ds}{(t-s)^{\frac 1 2+\frac d{2\rho}}}\right)\notag\\\le &Cg_c(t,y-x)\left(\left(t-(\tau_t^{{h}}-h)\vee t_1\right)^{\frac{\alpha}{2}}+h^{-\frac 1 q}\left(t-(\tau_t^{{h}}-h)\vee t_1\right)^{\frac{1}2-\frac d{2\rho}}\right)\notag\\
            \le& Cg_c(t,y-x)h^{\frac{{\alpha}} 2}.\label{THE_ALTERNATIVE_DELTA_2_FOR_ERR}
\end{align}

Let us conclude with the term $\Delta_t^1$ which can be used in a Gronwall type argument. Namely, set for $u\in (0,T] $:
$$f(u):=\sup_{(x,z)\in (\R^d) ^2}\frac{|\Gamma^h(0,x,u,z)-\Gamma(0,x,u,z)|}{g_c(u,x-z)}.
$$
 We know from \eqref{estigausschem} and \eqref{HK_AND_GRAD_SPACE_DENS} that $\sup_{s\in(0,T]}\  f(s)<+\infty$. Similarly to the proof of Lemma \ref{GAUSS_INT_LEMMA}, we write with $\frac 1{\bar q}=1-\frac 1 q$:
\begin{align*}
     |\Delta^1_t|&\le \int_0^t ds  f(s) \int_{\R^d} g_c(s,z-x)\ |b(s,z)|  \frac{g_c(t-s,y-z)}{(t-s)^{\frac 12}}dz\notag\\
      &\le Ct^{\frac{d}{2\rho} } \Bigg(\int_0^t ds\left(  \frac{f(s)}{s^{\frac{d}{2\rho}}(t-s)^{\frac 12+\frac{d}{2\rho}} } \right)^{\bar q}\Bigg)^{\frac 1{\bar q}}\|b\|_{L^q-L^p}g_c(t,y-x).
\end{align*}
With \eqref{decomperr}, \eqref{THE_ALTERNATIVE_DELTA_5_FOR_ERR}, \eqref{THE_ALTERNATIVE_DELTA_1_FOR_ERR}, \eqref{THE_ALTERNATIVE_DELTA_4_FOR_ERR}, \eqref{THE_ALTERNATIVE_DELTA_3_FOR_ERR} and \eqref{THE_ALTERNATIVE_DELTA_2_FOR_ERR}, we derive:
\begin{align*}
&f(t)\\
\le& C\left(h^{ \frac{\alpha} 2}\left(1+\left(\I_{\{\alpha=1\}}+\I_{\{(d,q)=(1,\infty)\}}\right)\ln\left(\frac Th\right)\right)+t^{\frac{d}{2\rho} } \left(\int_0^t   \frac{[f(s)]^{\bar q}ds}{s^{\bar q \frac{d}{2\rho}}(t-s)^{\bar q(\frac 12+\frac{d}{2\rho})}} \right)^{\frac 1{\bar q}}\right).
\end{align*}
Thus, up to an additional convexity inequality  if $ q<+\infty\iff \bar q>1$,  we get:
\begin{align*}
&[f(t)]^{\bar q}\\
\le& C^{\bar q}2^{\bar q-1}\left(h^{ \frac{\alpha\bar q} 2 }\left(1+\left(\I_{\{\alpha=1\}}+\I_{\{(d,q)=(1,\infty)\}}\right)\ln\left(\frac Th\right)\right)^{\bar q}+t^{\frac{d\bar q}{2\rho} } \int_0^t   \frac{[f(s)]^{\bar q}ds}{s^{\frac{d\bar q }{2\rho}}(t-s)^{\bar q(\frac 12+\frac{d}{2\rho})} } \right).
\end{align*}
It eventually follows from Lemma \ref{GR_VOL_LEMMA} applied with $\check\beta=\tilde\beta=\frac{d\bar q}{2\rho}$ and $\hat\beta=\bar q\left(\frac 1 2+\frac d{2\rho}\right)$ (by \eqref{COND_KR}, since $\frac d{2\rho}+\frac 1 q<\frac 12\Leftrightarrow \bar q\left(\frac 1 2+\frac d{2\rho}\right)<1$, one has $\hat\beta<1$ and $\tilde\beta+\hat\beta-1<\check\beta$) that
$$\sup_{t\in(0,T]}f(t)\le Ch^{\frac{\alpha}2}\left(1+\left(\I_{\{\alpha=1\}}+\I_{\{(d,q)=(1,\infty)\}}\right)\ln\left(\frac Th\right)\right),$$
which concludes the proof of Theorem \ref{MTHM} for the scheme \eqref{conteuler}. For the scheme \eqref{euler_BIS}, the conclusion holds without $\I_{\{(d,q)=(1,\infty)\}}$ in the right-hand side thanks to the improved estimation of the cutoff error $\bar\Delta^2_t$.

\section{Density estimates for the Euler scheme}\label{SEC_EUL_CTR}
This Section is dedicated to the proof of Proposition \ref{EST_DENS_SCHEME}.
\subsection{Existence of a transition density satisfying the Duhamel formula \eqref{DUHAMEL_SCHEME_PROP} and the Gaussian estimation \eqref{estigausschem}}\label{densduh}
For $k\in\leftB 0,n\rightB$ and $x\in\R^d$ let
$$X^h_t=x+(W_t-W_{t_k})+\int_{t_k}^t b_h\left(U_{\lfloor s/h\rfloor},X^h_{\tau^h_s}\right)ds,\quad t\in[t_k,T]$$
denote the Euler scheme started from $x$ at the discretization time $t_k=kh=\frac{kT}{n}$. We emphasize that the cutoffed drift coefficient $b_h$ defined in \eqref{cutoffb} coincides with $b$ as long as $|b|\le Bh^{-(\frac{1}{q}+\frac d{2\rho})}$ and is bounded from above by the threshold $Bh^{-(\frac{1}{q}+\frac d{2\rho})}$.
For $t\in (t_k,t_{k+1}]$, $X^h_t$ admits the density
\begin{align*}
\Gamma^h(t_k,x,t,y)&=\E[g_1\left(t-t_k,y-x-(t-t_k)  b_h(U_k,x)\right)]\\
&=\int_0^h \frac{ds}h g_1(t-t_k,y-x-(t-t_k)
  b_h(t_k+s,x))
  \end{align*}
with respect to the Lebesgue measure on $\R^d$. Since $z\mapsto g_1(t-t_k,z)$ is continuous and bounded by $(2\pi(t-t_k))^{-\frac{d}{2}}$, Lebesgue's theorem implies that $y\mapsto\Gamma^h(t_k,x,t,y)$ is continuous. Moreover, \eqref{estig} implies that

\begin{align}
\forall t\in(t_k,t_{k+1}],\;\forall x,y\in\R^d,\;\Gamma^h(t_k,x,t,y)\le c^{\frac{d}{2}}e^{\frac{B^2h^{\alpha}}{2(c-1)}}g_c(t-t_k,y-x).\label{estibrutunpas}
\end{align}
By the Markov structure of the Euler scheme, for $t\in (t_{k+1},T]$, 
\begin{align*}
  &\Gamma^h(t_k,x,t,y)\\
  =&\int_{(\R^d)^{\lceil\frac t h\rceil-k-1}}\Gamma^h(t_k,x,t_{k+1},z_1)\prod_{j=k+1}^{\lceil\frac t h\rceil-2}\Gamma^h(t_j,z_{j-k},t_{j+1},z_{j+1-k})\\
  &\hspace*{2cm}\times \Gamma^h(t_{\lceil\frac t h\rceil-1},z_{\lceil\frac t h\rceil-k-1},t,y)dz_1\cdots dz_{\lceil\frac t h\rceil-k-1}\\
\end{align*}
where, since $y\mapsto \Gamma^h(t_{\lceil\frac t h\rceil-1},z_{\lceil\frac t h\rceil-k-1},t,y)$ is continuous and bounded by $(2\pi(t-t_{\lceil\frac t h\rceil-1}))^{-\frac{d}{2}}$, Lebesgue's theorem implies that the left-hand side is a continuous function of $y$. Moreover, the last equality combined with \eqref{estibrutunpas} then Gaussian convolution imply that

  \begin{align*}
&\Gamma^h(t_k,x,t,y) \\
\le& \left(c^{\frac{d}{2}}e^{\frac{B^2{h^{\alpha}}}{2(c-1)}}\right)^{\lceil\frac t h\rceil-k}\int_{(\R^d)^{\lceil\frac t h\rceil-k-1}}g_c(t_{k+1}-t_k,z_1-x)\\
  & \times \prod_{j=k+1}^{\lceil\frac t h\rceil-2}g_c(t_{j+1}-t_j,z_{j+1-k}-z_{j-k})g_c(t-t_{\lceil\frac t h\rceil-1},y-z_{\lceil\frac t h\rceil-k-1})dz_1\cdots dz_{\lceil\frac t h\rceil-k-1}\\
  =&\left(c^{\frac{d}{2}}e^{\frac{B^2{h^{\alpha}}}{2(c-1)}}\right)^{\lceil\frac t h\rceil-k}g_c(t-t_k,y-x).
\end{align*}
The estimation \begin{equation}
   \Gamma^h(t_k,x,t,y)\le {c^{\frac{dT}{2h}}e^{\frac{B^2 Th^{\alpha-1}}{2(c-1)}}}g_c(t-t_k,y-x)\label{estibrut}
 \end{equation} is therefore valid for all $(x,y)\in (\R^d)^2$, $k\in\leftB 0,n\rightB$ and $t\in(t_k,T]$. Let us now check that the factor ${c^{\frac{dT}{2h}}e^{\frac{B^2 Th^{\alpha-1}}{2(c-1)}}}$ which goes to $+\infty$ when $h\to 0$ can be replaced by some finite constant not depending on the time-step $h$ and study the regularity of $\Gamma^h(t_k,x,t,y)$ in its forward variables $t$ and $y$.
 
 Let $t\in(t_k,T]$, $\varphi:\R^d\to\R^d$ be a $C^2$ function with compact support and $v(s,y)=\I_{\{s<t\}}g_1(t-s,.)\star \varphi(y)+\I_{\{s=t\}}\varphi(y)$. The function $v$ is bounded together with its spatial derivatives up to the order $2$ and its first order time derivative on the domain $[0,t]\times\R^d$ where it solves the heat equation
 \begin{equation*}
   \begin{cases}
      \partial_s v(s,y)+\frac{1}{2}\Delta v(s,y)=0,\;(s,y)\in[0,t]\times \R^d,\\v(t,y)=\varphi(y),\;y\in\R^d.
   \end{cases}
 \end{equation*}
 By It\^o's formula,
 $$\varphi(X^h_t)=v(t_k,x)+\int_{t_k}^t\nabla v(s,X^h_s).dW_s+\int_{t_k}^t \nabla v(s,X^h_s).b_h\left(U_{\lfloor s/h\rfloor},X^h_{\tau^h_s}\right)ds.$$
 Since $\nabla v$ and $b_h$ are bounded and, by \eqref{CTR_GRAD_GAUSS}, \eqref{estibrut} and Gaussian convolution, $\E[|\nabla g_1(t-s,X^h_{s}-y)|]\le C \frac{g_c(t-t_k,y-x)}{\sqrt{t-s}}$, taking the expectation and using Fubini's theorem, we deduce that
 \begin{align*}
 \int_{\R^d}\varphi(y)\Gamma^h(t_k,x,t,y)dy=&\int_{\R^d}\varphi(y)g_1(t-t_k,x-y)dy\\
 &+\int_{\R^d}\varphi(y)\int_{t_k}^t\E\left[b_h\left(U_{\lfloor s/h\rfloor},X^h_{\tau^h_s}\right).\nabla g_1(t-s,X^h_{s}-y)\right]ds    dy.
 \end{align*}
 Since $\varphi$ is arbitrary and $g_1$ is even in its spatial variable, we deduce that $dy$ a.e.,
 $$\Gamma^h(t_k,x,t,y)=g_1(t-t_k,y-x)-\int_{t_k}^t\E\left[b_h\left(U_{\lfloor s/h\rfloor},X^h_{\tau^h_s}\right).\nabla_y g_1(t-s,y-X^h_{s})\right]ds.$$
 This equality even holds for each $y\in \R^d$ since the left-hand side and the first term in the right-hand side are continuous functions of $y$ and in the derivation of \eqref{estischemspace} below we will check that the second term in the right-hand side satisfies the H\"older estimate in this inequality and is therefore also continuous in $y$.

The proof of Proposition \ref{EST_DENS_SCHEME} relies on this Duhamel formula where we expand $\Gamma^h$ around the Brownian semi-group. 
We could as well have considered the full parametrix expansion of the density of the scheme, used for instance in \cite{kona:mamm:02} or \cite{kona:kozh:meno:17}, but the \textit{one-step} Duhamel formulation is more consistant with the approach we also used to estimate the error of the Euler scheme.
We have, using that for $r\in[t_j,t_{j+1}]$, 
$X^h_r=X^h_{t_j}+W_r-W_{t_j}+b_h(U_{j},X^h_{t_j})(r-t_j)$, the independence between $X^h_{t_j},W_r-W_{t_j}$ and $U_{j}$ and the Gaussian semi-group property for the second equality,
\begin{align}
  &\Gamma^h(t_k,x,t,y)\notag\\
  =&g_1(t-t_k,y-x)-\sum_{j=k}^{\lceil\frac t h\rceil-1}\int_{t_j}^{t_{j+1}\wedge t}\E\left[b_h(U_{j},X^h_{t_j})\cdot\nabla_y g_1(t-r,y-X^h_r)\right]dr\notag\\
  =&g_1(t-t_k,y-x)\notag\\
  &-\sum_{j=k}^{\lceil\frac t h\rceil-1}\frac 1 h\int_{r=t_j}^{t_{j+1}\wedge t}\int_{s=t_j}^{t_{j+1}}\E\left[b_h(s,X^h_{t_j})\cdot\nabla_y g_1(t-t_j,y-X^h_{t_j}-b_h(s,X^h_{t_j})(r-t_j))\right]dsdr\notag\\
=&g_1(t-t_k,y-x)-\frac{1}{h}\int_{r=t_k}^{t_{k+1}\wedge t}\int_{s=t_k}^{t_{k+1}}b_h(s,x)\cdot\nabla_y g_1(t-t_k,y-x-b_h(s,x)(r-t_k))dsdr\notag\\
  &-\sum_{j=k+1}^{\lceil\frac t h\rceil-1}\frac{1}{h}\int_{r=t_j}^{t_{j+1}\wedge t}\int_{s=t_j}^{t_{j+1}}\int_{\R^d}\Gamma^h(t_k,x,t_j,z)\notag\\
  &\hspace*{3cm}\times b_h(s,z)\cdot\nabla_yg_1(t-t_j,y-z-b_h(s,z)(r-t_j))dzdsdr.\label{duhamgauschem}
\end{align}
Since, by \eqref{CTR_GRAD_GAUSS}, $\forall c>1,\ \exists C<\infty,\;\forall u\in (0,T],\forall x\in\R^d,\;|\nabla g_1(u,x)|\le \frac{C}{\sqrt{u}}g_{\frac{1+c}{2}}(u,x)$ applying \eqref{majodiff} with $c'=\frac{2c}{1+c}$, we obtain that
\begin{align}
  \exists C<\infty,\; \forall u\in (0,T],\;\forall u'\in[0,u\wedge h],\;\forall s\in [0,T],\;\forall (x,y)\in(\R^d)^2,\notag\\
  \;|\nabla g_1(u,y-x-b_h(s,x)u')|\le C \frac{g_c(u,y-x)}{\sqrt{u}}.\label{majogradgschem}
\end{align}
Set $m_{k,j}=\sup_{(x,z)\in(\R^d)^2}\frac{\Gamma^h(t_k,x,t_j,z)}{g_c(t_j-t_k,z-x)}$. 
It is clear from \eqref{estibrut} that $m_{k,j}<+\infty $.
Using that  $\|b_h\|_\infty\le Bh^{-(\frac{1}{q}+\frac d{2\rho})}$,  H\"older's inequality in space for the second inequality (see also Equation \eqref{BD_G_EPS} in the proof of Lemma \ref{GAUSS_INT_LEMMA}), and then that $t_\ell-t_k\ge h$ and $(t_j-t_k)\ge \frac 12(s-t_k)$ for $s\in[t_j,t_{j+1}]$ with $j\ge k+1$ for the third one, and eventually H\"older's inequality in time for the fourth inequality, we deduce that for $\ell\in\leftB k+1,n\rightB$,
\begin{align*}
  &
  \frac{|\Gamma^h(t_k,x,t_\ell,y)|}{g_c(t_\ell-t_k,y-x)}\\
  \le& c^{\frac{d}{2}}+\frac{CBh^{1-(\frac 1q+\frac d{2\rho})}}{\sqrt{t_\ell-t_k}}\\
  &+\sum_{j=k+1}^{\ell-1}\frac{m_{k,j}}{g_c(t_\ell-t_k,y-x)}\int_{t_j}^{t_{j+1}}\int_{\R^d}g_c(t_j-t_k,z-x)|b_h(s,z)|\frac{g_c(t_\ell-t_j,y-z)}{\sqrt{t_\ell-t_j}}dzds\\
                                                        \le &c^{\frac{d}{2}}+\frac{CBh^{\frac 12+\frac{\alpha}2}}{\sqrt{t_\ell-t_k}}+C\sum_{j=k+1}^{\ell-1}\frac{m_{k,j}(t_\ell-t_k)^{\frac{d}{2\rho}}}{(t_j-t_k)^{\frac{d}{2\rho}}(t_\ell-t_j)^{\frac 12+\frac{d}{2\rho}}}\int_{t_j}^{t_{j+1}}\|b(s,.)\|_{L^\rho}ds\\
                                                        \le& C+C\max_{j=k+1}^{\ell-1}m_{k,j}(t_\ell-t_k)^{\frac{d}{2\rho}}\int_{t_k}^{t_\ell}\frac{\|b(s,.)\|_{L^\rho}ds}{(s-t_k)^{\frac{d}{2\rho}}(t_\ell-s)^{\frac 12+\frac{d}{2\rho}}}\\
  \le &C+C\max_{j=k+1}^{\ell-1}m_{k,j}\|b\|_{L^q-L^\rho}(t_{\ell}-t_k)^{{\frac{\alpha}2}}.
\end{align*}
Taking the supremum over $(x,y)\in(\R^d)^2$ and remarking that the right-hand side is non-decreasing with $\ell$, we deduce that 
$$\max_{j=k+1}^\ell m_{k,j}\le C+C\|b\|_{L^q-L^\rho}(t_{\ell}-t_k)^{{\frac{\alpha}2}}\max_{j=k+1}^\ell m_{k,j}.$$ Hence when $t_\ell-t_k\le \theta:=\left(2C\|b\|_{L^q-L^\rho}\right)^{-{\frac 2{\alpha}}}$, then $\max_{j=k+1}^\ell m_{k,j}\le 2C$. Let us now assume that $h\le\theta$ so that $\max_{\kappa=1}^{n-1}\max_{j=\kappa+1}^{(\kappa+\lfloor\frac\theta h\rfloor)\wedge n} m_{\kappa,j}\le 2C$. This gives the Gaussian estimate for the density of the scheme, independently of $h$, provided the associated time interval is small enough but at a \textit{macro} scale. For an arbitrary \textit{macro} time interval the idea is now to chain the previous estimates. Assuming that $t-t_k\ge \theta $ and
setting $J=\lceil\frac{t-t_k}{\tau^h_\theta}\rceil-1$, we have, under the convention $y_0=x$,
\begin{align*}
&\Gamma^{h}(t_k,x,t,y)\\
=&\int_{(\R^d)^{J}}\prod_{j=1}^{J} \Gamma^{h}(t_k+(j-1)\tau^h_\theta,y_{j-1},t_k+j\tau^h_\theta,y_j)\Gamma^{h}(t_k+J\tau^h_\theta,y_{J},t,y)dy_1\ldots dy_J.
\end{align*}
Since $t-(t_k+J\tau^h_\theta)\le \tau^h_\theta$, when $t$ does not belong to the discretization grid $\{t_j=jh:{j\in\leftB 0, n\rightB}\}$, combining the just derived bound and \eqref{estibrutunpas}, we get 
\begin{align*}
\Gamma^{h}(t_k+J\tau^h_\theta,y_{J},t,y)=&\int_{\R^d}\Gamma^{h}(t_k+J\tau^h_\theta,y_{J},\tau^h_t,z)\Gamma^{h}(\tau^h_t,z,t,y)dz\\
\le& 2Cc^{\frac{d}{2}}e^{\frac{B^2{h^{\alpha}}}{2(c-1)}}g_c(t-(t_k+J\tau^h_\theta),y_J-y),
\end{align*}
and the same estimation holds without the factor $c^{\frac{d}{2}}e^{\frac{B^2{h^{\alpha}}}{2(c-1)}}$ when $t$ belongs to the discretization grid. Hence, when $h\le\theta$ which implies $\tau^h_\theta>\frac\theta 2$, proceeding similarly to the proof of \eqref{estibrut} (Gaussian chaining argument) with $h$ replaced by $\tau_\theta^h$, we derive:
\begin{align*}
\forall 0\le k<n,\;\forall t\in (t_k,T],\;\Gamma^h(t_k,x,t,y)&\le (2C)^{\lceil \frac{t-t_k}{\tau^h_\theta}\rceil}c^{\frac{d}{2}}e^{\frac{B^2{h^{\alpha}}}{2(c-1)}}g_c(t-t_k,y-x)\\
&\le (2C)^{1+2\frac{t-t_k}{\theta}}c^{\frac{d}{2}}e^{\frac{B^2{h^{\alpha}}}{2(c-1)}}g_c(t-t_k,y-x).
\end{align*}
This gives the first estimation \eqref{estigausschem} in the proposition. 

Similar estimates with the factors $e^{\frac{B^2 h^{\alpha}}{2(c-1)}}$ replaced by $e^{\frac{B^2 }{2(c-1)}}$ and with $e^{\frac{B^2 Th^{\alpha-1}}{2(c-1)}}$ replaced by $e^{\frac{B^2 T}{2(c-1)h}}$ in \eqref{estibrut} can be derived for the scheme $\bar X^h $ defined in \eqref{euler_BIS} (and even for the scheme with the same cutoff when the cutoffed drift is kept on the first time-step).
\subsection{H\"older regularity of the transition density in the forward time variable}
We now prove \eqref{estischemtemps}. Let $1\le k<\ell<n$, $x,y\in\R^d$ and $t\in[t_\ell,t_{\ell+1}]$. We want to estimate $\Gamma^h(t_k,x,t_\ell,y)-\Gamma^h(t_k,x,t,y)$, which, according to \eqref{duhamgauschem}, is equal to $\Delta^1+\Delta^2+\Delta^3+\Delta^4$ with
\begin{align*}
  \Delta^1&=g_1(t_\ell-t_k,y-x)-g_1(t-t_k,y-x),\\
  \Delta^2&=\frac{1}{h}\int_{r=t_k}^{t_{k+1}}\int_{s=t_k}^{t_{k+1}}b_h(s,x)\cdot[\nabla g_1(t-t_k,w)-\nabla g_1(t_\ell-t_k,w)]|_{w=y-x-b_h(s,x)(r-t_k)}dsdr,\\
  \Delta^3&=\sum_{j=k+1}^{\ell-1}\frac{1}{h}\int_{r=t_j}^{t_{j+1}}\int_{s=t_j}^{t_{j+1}}\int_{\R^d}\Gamma^h(t_k,x,t_j,z)b_h(s,z)\cdot[\nabla g_1(t-t_j,w)\notag\\
  &\hspace*{2cm}-\nabla g_1(t_\ell-t_j,w)]|_{w=y-z-b_h(s,z)(r-t_j)}dzdsdr,\\
   \Delta^4&=\frac 1 h\int_{r=t_\ell}^{t}\int_{s=t_\ell}^{t_{\ell+1}}\int_{\R^d}\Gamma^h(t_k,x,t_\ell,z)b_h(s,z)\cdot\nabla_yg_1(t\!-\!t_\ell,y\!-\!z\!-\!b_h(s,z)(r\!-\!t_\ell))dzdsdr.
\end{align*}
{From} \eqref{estigausschem}, \eqref{majogradgschem}, {recalling that $|b_h|\le |b| $}, {applying then} H\"older's inequality in space, then in time and {using lastly} that $t-t_k\le 2(t_\ell-t_k)$ and $t-t_\ell\le h$, we obtain that
\begin{align*}
  |\Delta^4|&\le \frac{C(t-t_\ell)} h\int_{t_\ell}^{t_{\ell+1}}\int_{\R^d}g_c(t_\ell-t_k,z-x)|b(s,z)|\frac{g_c(t-t_\ell,y-z)}{\sqrt{t-t_\ell}}dzds\\
            &\le \frac{C(t-t_\ell)(t-t_k)^{\frac{d}{2\rho}}}{h(t_\ell-t_k)^{\frac{d}{2\rho}}(t-t_\ell)^{\frac 1 2+\frac{d}{2\rho}}}\int_{t_\ell}^{t_{\ell+1}}\|b(s,.)\|_{L^\rho}ds g_c(t-t_k,y-x)\\
  &\le \frac{C(t-t_\ell)^{\frac 1 2-\frac{d}{2\rho}}}{h}h^{1-\frac 1 q}\|b\|_{L^q-L^\rho}g_c(t-t_k,y-x)\\\
  &\le C(t-t_\ell)^{{\frac{\alpha} 2}}\|b\|_{L^q-L^\rho}g_c(t-t_k,y-x).
\end{align*}
Reasoning like in the above derivation of \eqref{majogradgschem} with \eqref{DIFF_GRAD_GAUSS_LEM_TIME} replacing \eqref{CTR_GRAD_GAUSS}, we obtain  the existence of a finite constant $C$ such that for all $0<u<u'\le T$, all $u''\in[0,u\wedge h]$, all $s\in[0,T]$ and all $(y,z)\in(\R^d)^2$,
\begin{align}
 & |\nabla g_1(u,y-z-b_h(s,x)u'')-\nabla g_1(u',y-z-b_h(s,z)u'')|
\notag\\  \le& C \frac{|u'-u|\wedge u}{u^{\frac{3}{2}}}(g_c(u,y-z)+g_c(u',y-z))\notag\\\le& C \frac{|u'-u|\wedge u}{u^{\frac{3}{2}}}g_c(u',y-z)\mbox{ if }u'\le 2 u.\label{majodifgradgschem}
\end{align}
This inequality, together with $|b_h|\le Bh^{-(\frac 1q+\frac{d}{2\rho})}$ and $t_\ell-t_k\ge h$, implies that for $\check\alpha\in (0,2]$,
\begin{equation}
   |\Delta^2|
   \le
   C\frac{(t-t_\ell)^{\frac \alpha 2}}{(t_\ell-t_k)^{\frac{1+\alpha}{2}}}g_c(t-t_k,y-x)\int_{t_k}^{t_{k+1}}|b_h(s,x)|ds\le C\frac{(t-t_\ell)^{\frac {\check\alpha} 2}}{(t_\ell-t_k)^{\frac{\check\alpha}{2}}}{h^{\frac{\alpha}{2}}}g_c(t-t_k,y-x).\label{estidel2}
\end{equation}
The estimation of $\Delta_3$ is a bit more involved. We suppose that $\ell\ge k+2$ since $\Delta_3=0$ otherwise.
Let $\tilde \alpha\in[\alpha,1]$.
Using \eqref{estigausschem}, \eqref{majodifgradgschem} and $|b_h|\le |b|$, H\"older's inequality in space then that $s-t_k\le 2(t_j-t_k)$ when $s\in[t_j,t_{j+1}]$ with $j\ge k+1$ and $(t_{\ell-1}-t_k)> \frac{t-t_k}{3}$ and last H\"older's inequality in time,  we obtain that
\begin{align*}
  |\Delta^3|\le &C\Big(\sum_{j=k+1}^{\ell-1}\int_{t_j}^{t_{j+1}}\int_{\R^d}g_c(t_j-t_k,z-x)|b(s,z)|\frac{g_c(t-t_j,y-z)(t-t_\ell)^{\frac {{\tilde \alpha}} 2}}{(t_\ell-t_j)^{\frac{1+{\tilde \alpha}}{2}}}dzds\\
 \le &C(t-t_\ell)^{\frac {{\tilde \alpha}} 2}(t-t_k)^{\frac d{2\rho}}g_c(t-t_k,y-x)\sum_{j=k+1}^{\ell-1}\frac{\int_{t_j}^{t_{j+1}}\|b(s,.)\|_{L^\rho}ds}{(t_j-t_k)^{\frac{d}{2\rho}}(t_\ell-t_j)^{\frac{1+{\tilde \alpha}}{2}+\frac{d}{2\rho}}}\\
  \le& C(t-t_\ell)^{\frac {{\tilde \alpha}} 2}g_c(t-t_k,y-x)\left[(t-t_k)^{\frac d{2\rho}}\int_{t_k}^{t_{{\ell-1}}}\frac{\|b(s,.)\|_{L^\rho}}{(s-t_k)^{\frac{d}{2\rho}}(t_\ell-s)^{\frac{1+\tilde \alpha}{2}+\frac{d}{2\rho}}}ds\right.\\
  &\left.+{h^{-(\frac{1+{\tilde \alpha}}{2}+\frac{d}{2\rho})}\int_{t_{\ell-1}}^{t_\ell}\|b(s,.)\|_{L^\rho} ds}\right]\\
            \le& C(t-t_\ell)^{\frac {{\tilde \alpha}} 2}g_c(t-t_k,y-x)\bigg[(t_\ell-t_k)^{-\frac{1+\tilde \alpha}{2}}\int_{t_k}^{\frac{t_k+t_\ell}{2}}\frac{\|b(s,.)\|_{L^\rho}ds}{(s-t_k)^{\frac{d}{2\rho}}}\\ &+\int^{t_{\ell-1}}_{\frac{t_k+t_\ell}{2}}\frac{\|b(s,.)\|_{L^\rho}ds}{(t_\ell-s)^{\frac{1+\tilde \alpha}{2}+\frac{d}{2\rho}}}
            +{h^{-(\frac{1+{\tilde \alpha}}{2}+\frac{d}{2\rho})}\int_{t_{\ell-1}}^{t_\ell}\|b(s,.)\|_{L^\rho} ds}\bigg]\\
            \le& C(t-t_\ell)^{\frac {{\tilde \alpha} }2}g_c(t-t_k,y-x){\|b\|_{L^q-L^\rho}}\\
            &\times\left[ (t_\ell-t_k)^{{\frac{\alpha-{\tilde \alpha}} 2}}+\I_{\{\tilde\alpha>\alpha\}}h^{\frac{\alpha-\tilde \alpha}{2}}+\I_{\{\tilde\alpha=\alpha\}}\left(\ln\left(\frac{t_\ell-t_k}{2h}\right)\right)^{1-\frac 1 q}+h^{\frac{\alpha-\tilde \alpha}{2}}\right].\end{align*}
          When $\alpha<1$, we may choose $\tilde\alpha\in(\alpha,1]$ while when $\alpha=1$ (i.e. $\rho=q=\infty$) the only possible choice is $\tilde\alpha=1$.
          We conclude that
          $$|\Delta^3|\le C(t-t_\ell)^{\frac {{\alpha} }2}g_c(t-t_k,y-x)\left(1+\I_{\{\alpha=1\}}\ln\left(\frac{t_\ell-t_k}{h}\right)\right).$$
     Using \eqref{DIFF_GRAD_GAUSS_LEM_TIME} to deal with $\Delta^1$, we conclude that
     \eqref{estischemtemps} holds. Similar estimates with $h^{\frac{\alpha}{2}}$ replaced by $1$ in the right-hand side of \eqref{estidel2} can be derived for the scheme $\bar X^h $ defined in \eqref{euler_BIS} (and even for the scheme with the same cutoff when the cutoffed drift is kept on the first time-step).

     \subsection{H\"older regularity of the transition density in the forward spatial variable}
 Let us now suppose that $\alpha<1$ and prove \eqref{estischemspace}.
First of all, by \eqref{estigausschem}, 
\begin{align*}
|\Gamma^h(t_k,x,t,y)-\Gamma^h(t_k,x,t,y')|\le & \Gamma^h(t_k,x,t,y)+\Gamma^h(t_k,x,t,y')\\
\le& C\Big({g}_c(t-t_k,y-x)+{ g}_c(t-t_k,y'-x)\Big),
\end{align*}
so that \eqref{estischemspace} holds in  the \textit{global off-diagonal regime} $|y-y'|^2>(t-t_k)/4$ where $\frac{|y-y'|^{\alpha}\wedge (t-t_k)^{\frac{\alpha}{2}}}{(t-t_k)^{\frac{\alpha}{2}}}> 4^{-\frac{\alpha}{2}}$. Therefore, it is enough to focus on the so-called \textit{global diagonal regime} 
\begin{equation}\label{DIAG}
|y-y'|^2\le (t-t_k)/4
\end{equation}
where we set $u=t-|y-y'|^2$.
By the Duhamel formula \eqref{DUHAMEL_SCHEME_PROP}, we have
\begin{align*}
 \Gamma^h(t_k,x,t,y)&-\Gamma^h(t_k,x,t,y')=g_1(t-t_k,y-x)-g_1(t-t_k,y'-x)+T_{1}+T_{2}+T_3\mbox{ where }\\
 &T_1=\int_{t_k}^{t_{k+1}\wedge t}{\cal E}^h_rdr+\I_{\{t>t_{k+1}\}}\int_{(\tau^h_t-h)\vee t_{k+1}}^{t}{\cal E}^h_rdr,\\&T_2=\I_{\{\tau^h_t-h>t_{k+1},u\ge \tau^h_t-h\}}\int_{t_{k+1}}^{\tau^h_t-h}{\cal E}^h_rdr+\I_{\{u<\tau^h_t-h\}}\int^{\tau^h_u}_{t_{k+1}}{\cal E}^h_rdr,\\&T_3=\I_{\{u<\tau^h_t-h\}}\int_{\tau^h_u}^{\tau^h_t-h}{\cal E}^h_rdr,\\
\mbox{ with }&{\cal E}^h_r:=\E\left[b_h(U_{\lfloor \frac rh\rfloor },x)\cdot\left(\nabla_{y'} g_1(t-r,y'-X^h_r)-\nabla_y g_1(t-r,y-X^h_r)\right)\right]dr.
\end{align*}
Note that when $u<\tau^h_t-h$, then $|y-y'|^2>h$ so that, in view of \eqref{DIAG}, $t>t_{k+4}$ and $(\tau^h_t-h)\wedge u\ge t_{k+3}$. By \eqref{DIFF_GRAD_GAUSS_LEM} and since $\alpha\le 1$, we first get 
\begin{align*}
&|g_1(t-t_k,y-x)- g_1(t-t_k,y'-x)|\\
\le& C\frac{|y-y'|^{\alpha}\wedge (t-t_k)^{\frac{\alpha}{2}}}{(t-t_k)^{\frac{\alpha}{2}}}\left(g_c(t-t_k,y-x)+g_c(t-t_k,y'-x)\right).
\end{align*}
Using \eqref{DIFF_GRAD_GAUSS_LEM}, \eqref{estigausschem} and $|b_h|\le Bh^{-\left(\frac 1 q+\frac d{2\rho}\right)}$, then Gaussian convolution, we then obtain that

\begin{align*}
   {\mathcal E}^h_r&\le C\int_{\R^d}g_c(r-t_k,z-x)\frac{|y-y'|^{\alpha}}{h^{\frac 1 q+\frac d{2\rho}}(t-r)^{\frac{1+\alpha}{2}}}\left(g_c(t-r,y'-z)+g_c(t-r,y-z)\right)dz\\&\le C\left(g_c(t-t_k,y-x)+g_c(t-t_k,y'-x)\right)\frac{|y-y'|^{\alpha}}{h^{\frac 1 q+\frac d{2\rho}}(t-r)^{\frac{1+\alpha}{2}}}.
\end{align*}
Therefore, when $y'\neq y$,
\begin{align*}
  &\frac{|T_1|}{\left(g_c(t-t_k,y-x)+g_c(t-t_k,y'-x)\right)|y-y'|^{\alpha}}\\
  \le& \frac{C}{h^{\frac 1 q+\frac d{2\rho}}}\left(\int_{t_k}^{t_{k+1}\wedge t}\!\!\!\!\!\frac{dr}{(t-r)^\frac{1+\alpha}{2}}+\I_{\{t>t_{k+1}\}}\int_{(\tau^h_t-h)\vee t_{k+1}}^{t}\frac{dr}{(t-r)^{\frac{1+\alpha}{2}}}\right)\\
  \le& \frac{C}{h^{\frac 1 q+\frac d{2\rho}}}\left(\I_{\{t-t_k\le 2h\}}\frac{(t-t_k)^{\frac{1-\alpha}{2}}}{(1-\alpha)/2}\right.\\
  &\left.+\I_{\{t-t_k> 2h\}}\frac{h}{((t-t_k)/2)^{\frac{1+\alpha}{2}}}+\I_{\{t>t_{k+1}\}}\frac{(t-((\tau^h_t-h)\vee t_{k+1}))^{\frac{1-\alpha}{2}}}{(1-\alpha)/2}\right)\le C.\end{align*}

For the scheme $\bar X^h $ defined in \eqref{euler_BIS} (and even for the scheme with the same cutoff when the cutoffed drift is kept on the first time-step), the constant $C$ in the right-hand side should be replaced by $C(t-t_k)^{-\frac{\alpha} 2}$, where the denominator does not prevent from deriving \eqref{estischemspace}. The forthcoming estimations of $T_2$ and $T_3$ rely on the bound $|b_h|\le|b|$ and are valid for the two schemes.                        

Let ${\tilde \alpha}\in (\alpha,1]$. Using $|b_h|\le |b|$, \eqref{estigausschem}, \eqref{estig} and \eqref{DIFF_GRAD_GAUSS_LEM}, then Gaussian convolution, we obtain that  
\begin{align*}
  {\mathcal E}^h_r&\le \frac{C}{h}\int_{\tau^h_r}^{\tau^h_r+h}\int_{\R^d\times\R^d}g_c(\tau^h_r-t_k,z-x)g_c(r-\tau^h_r,w-z)|b(s,z)|\frac{|y-y'|^{\tilde \alpha}}{(t-r)^{\frac{1+{\tilde \alpha}} 2}}\\
  &\hspace*{2.6cm}\times\left(g_c(t-r,y'-w)+g_c(t-r,y-w)\right)dz dwds\\
  &\le \frac{C}{h}\int_{\tau^h_r}^{\tau^h_r+h}\int_{\R^d}g_c(\tau^h_r-t_k,z-x)|b(s,z)|\frac{|y-y'|^{\tilde \alpha}}{(t-r)^{\frac{1+{\tilde \alpha}} 2}}\\
  &\hspace*{2.6cm}\times \left(g_c(t-\tau^h_r,y'-z)+g_c(t-\tau^h_r,y-z)\right)dzds.
\end{align*}
Let us assume that $\tau^h_t-h>t_{k+1}$ (so that $\tau^h_t-h-t_k\ge \frac{t-t_k}{2}$) and set $\ell=\I_{\{u\ge \tau^h_t-h\}}(\lfloor\frac t h\rfloor -1)+\I_{\{u<\tau^h_t-h\}}\lfloor\frac u h\rfloor$. Using that $t-r\ge \frac{t-\tau^h_r}{2}$ for $r\le \tau^h_t-h$ then H\"older's inequality in space, we deduce that when $y'\ne y$
\begin{align*}
 &  \frac{|T_2|}{|y-y'|^{\tilde \alpha}} \\
   \le &\frac C{h}\sum_{j=k+1}^{\ell -1}\int_{t_j}^{t_{j+1}}\int_{t_j}^{t_{j+1}}\int_{\R^d}g_c(t_j-t_k,z-x)\frac{|b(s,z)|}{(t-t_j)^{\frac{1+{\tilde \alpha}} 2}}\\
   &\hspace*{2cm}\times\left(g_c(t-t_j,y'-z)+g_c(t-t_j,y-z)\right)dzdsdr\\
                                    \le & C\left(g_c(t-t_k,y'-x)+g_c(t-t_k,y-x)\right)\sum_{j=k+1}^{\ell -1}\int_{t_j}^{t_{j+1}}\frac{(t-t_k)^{\frac d{2\rho}}\|b(s,.)\|_{L^\rho}}{(t_j-t_k)^{\frac d{2\rho}}(t-s)^{\frac{1+{\tilde \alpha}} 2+\frac d{2\rho}}}ds.\end{align*}
 Since $t_j-t_k\ge \frac{s-t_k} 2$ for $s\in[t_j,t_{j+1}]$ with $j\ge k+1$ and when $j<\frac{k-1+\lfloor\frac t h\rfloor}{2}$, $t_{j+1}\le \frac{t_k+\tau^h_t}{2}$ so that $t-t_{j+1}\ge\frac{t-t_k}{2}$ while when $j\ge \frac{k-1+\lfloor\frac t h\rfloor}{2}$, $t_j-t_k\ge \frac{\tau^h_t-h-t_k}{2}\ge \frac{t-t_k}{4}$, the last sum is smaller than
                                  \begin{align*}
         &\frac{2^{\frac{1+{\tilde \alpha}} 2+\frac d{\rho}}}{(t-t_k)^{\frac{1+{\tilde \alpha}} 2}}\sum_{j=k+1}^{\ell -1}\I_{\{j<\frac{k-1+\lfloor\frac t h\rfloor}{2}\}}\int_{t_j}^{t_{j+1}}\frac{\|b(s,.)\|_{L^\rho}}{(s-t_k)^{\frac d{2\rho}}}ds\\
          &+4^{\frac d{2\rho}}\sum_{j=k+1}^{\ell -1}\I_{\{j\ge \frac{k-1+\lfloor\frac t h\rfloor}{2}\}}\int_{t_j}^{t_{j+1}}\frac{\|b(s,.)\|_{L^\rho}}{(t-s)^{\frac{1+{\tilde \alpha}} 2+\frac d{2\rho}}}ds\\
          \le &C\|b\|_{L^q-L^\rho}\frac{(\frac{t_k+\tau^h_t}{2}-t_k)^{\frac{1+\alpha} 2}}{(t-t_k)^{\frac{1+{\tilde \alpha}} 2}}+C\|b\|_{L^q-L^\rho}(t-t_\ell)^{\frac{\alpha-{\tilde \alpha}}{2}},                       
                                  \end{align*}
                                  where we used H\"older's inequality in time for the last inequality. Since $t-t_\ell\ge t-u=|y-y'|^2$ and $|y-y'|^2\le\frac{t-t_k}{4}$, we deduce that
                                  \begin{align*}
                                   |T_2|\le C|y-y'|^{\alpha}\left(g_c(t-t_k,y'-x)+g_c(t-t_k,y-x)\right).
                                  \end{align*}

Using $|b_h|\le |b|$, \eqref{estigausschem}, \eqref{estig} and \eqref{CTR_GRAD_GAUSS}, then Gaussian convolution, we obtain that
\begin{align*}
  {\mathcal E}^h_r
  &\le \frac{C}{h}\int_{\tau^h_r}^{\tau^h_r+h}\int_{\R^d}g_c(\tau^h_r-t_k,z-x)\frac{|b(s,z)|}{(t-r)^{\frac 1 2}}\left(g_c(t\!-\!\tau^h_r,y'\!-\!z)+g_c(t\!-\!\tau^h_r,y\!-\!z)\right)dzds.
\end{align*}

When $u<\tau^h_t-h$ which implies that $t-u>h$, using $t-r\ge \frac{t-\tau^h_r}{2}$ for $r\le \tau^h_t-h$, 
 then H\"older's inequality in space and last that $\tau^h_u-t_k\ge u-h-t_k\ge t-t_k-2(t-u)=t-t_k-2|y-y'|^2\ge \frac{t-t_k}{2}$ and H\"older's inequality in time, we deduce that
\begin{align*}
  |T_3|
            &\le C\left(g_c(t-t_k,y'-x)+g_c(t-t_k,y-x)\right)\sum_{j=\lfloor \frac u h\rfloor}^{\lfloor \frac t h\rfloor-2}\int_{t_j}^{t_{j+1}}\frac{(t-t_k)^{\frac d{2\rho}}\|b(s,.)\|_{L^\rho}}{(t_j-t_k)^{\frac d{2\rho}}(t-s)^{\frac 1 2+\frac d{2\rho}}}ds\\
              &\le C\left(g_c(t-t_k,y'-x)+g_c(t-t_k,y-x)\right)\|b\|_{L^q-L^\rho}(t-\tau^h_u)^{\frac{\alpha}{2}},
\end{align*}
where, by definition of $u$ and since $u<\tau^h_t-h$, $(t-\tau^h_u)^{\frac{\alpha}{2}}\le (|y'-y|^2+h)^{\frac{\alpha}{2}}<(2|y-y'|^2)^{\frac{\alpha}{2}}$.

\section{Density estimates for the diffusion}\label{Sec-Diff}
The section is devoted to the proof of Proposition \ref{THM_HK_CONT}. 
We focus without loss of generality on the case $\alpha<1 $, which in particular implies that either $\rho$ or $q$ is finite.  Indeed,
  when $b\in L^\infty-L^\infty$ i.e. $\alpha=1 $, since we are considering a compact time interval, $b$ also belongs to $L^q-L^\infty$ for each $q\in (2,\infty)$. Notice that, for a time-space bounded drift, the estimates of the proposition are known, 
see e.g. \cite{meno:pesc:zhan:21}.

Using $|b_h|\le Bh^{-\left(\frac 1 q+\frac d {2\rho}\right)}\wedge|b|$, the independence between $U_{\lfloor \frac s h\rfloor}$ and $X^h_{\tau^h_s}$, then \eqref{estigausschem}, H\"older's inequality in space combined with $\|g_c(t,.)\|_{L^{\rho'}}={\rho'}^{-\frac{d}{2\rho'}}(2\pi c t)^{-(1-\frac 1{\rho'})\frac d 2}$, that $t_k\ge \frac s 2$ for $s\in[t_k,t_{k+1}]$ with $k\ge 1$ and last H\"older's inequality in time, we obtain
 \begin{align}
   &\E\left[\int_0^T\left|b_h\left(U_{\lfloor \frac s h\rfloor},X^h_{\tau^h_s}\right)\right|^2ds\right]\notag\\
    \le&  B^2h^{\alpha}+\E\left[\int_h^T\left|b\left(s,X^h_{\tau^h_s}\right)\right|^2ds\right]\notag\\
                                                                                                       \le &B^2h^{\alpha}+\sum_{k=1}^{n-1}\int_{t_k}^{t_{k+1}}\int_{\R^d}|b(s,y)|^2g_c(t_k,y-x)dyds\notag\\\le &B^2h^{\alpha}+C\sum_{k=1}^{n-1}\int_{t_k}^{t_{k+1}}\|b(s,.)\|^2_{L^\rho}\|g_c(t_k,.)\|_{L^{\frac{\rho}{\rho-2}}}ds\notag\\\le & B^2h^{\alpha}+C\sum_{k=1}^{n-1}t_k^{-\frac d\rho}\int_{t_k}^{t_{k+1}}\|b(s,.)\|^2_{L^\rho}ds\le B^2h^{\alpha}+C\int_0^Ts^{-\frac d \rho}\|b(s,.)\|^2_{L^\rho}ds\notag \\
   \le & B^2T^{\alpha}+C\|b\|_{L^q-L^\rho}^2T^{\alpha}.\label{majocar}
 \end{align}
Since, by the Cauchy-Schwarz inequality,
$$\forall 0\le u\le t\le T,\;\left|\int_u^tb_h\left(U_{\lfloor \frac s h\rfloor},X^h_{\tau^h_s}\right)ds\right|\le (t-u)^{1/2}\left(\int_0^T\left|b_h\left(U_{\lfloor \frac s h\rfloor},X^h_{\tau^h_s}\right)\right|^2ds\right)^{1/2},$$
with the Ascoli-Arzel\`a theorem, we deduce the tightness of the laws of the continuous processes $$\left(\int_0^tb_h\left(U_{\lfloor \frac s h\rfloor},X^h_{\tau^h_s}\right)ds\right)_{t\in[0,T]}$$ indexed by $h=\frac T n$ with $n\in\N^*$, when the space $\cal C$ of continuous functions from $[0,T]$ to $\R^d$ is endowed with the supremum norm. With the continuity of the sum on this space, we deduce that the laws $P^h$ of $X^h$ are tight. We may extract a subsequence still denoted by $(P^h)$ for notational simplicity such that $P^{h}$ weakly converges to some limit $P$ as $h\to 0$. For fixed $t\in(0,T]$, the weak convergence of $P^{h}_t(dy)=\Gamma^h(0,x,t,y)dy$ to $P_t(dy)$ together with \eqref{estischemspace},  \eqref{estigausschem} and the Ascoli-Arzel\`a theorem, ensure that $P_t(dy)=\Gamma(0,x,t,y)dy$ for some function $\Gamma$ satisfying \eqref{HKholdspace} and \eqref{HK_AND_GRAD_SPACE_DENS} .

Let $\varphi:\R^d\to\R$ be a $C^2$ function with compact support, $\psi:(\R^d)^p\to\R$ be continuous and bounded, $0\le s_1\le s_2\le\ldots\le s_p\le u\le t\le T$ with $u>0$ and $F$ denote the functional on ${\cal C}$ defined by
$$F(\xi)=\left(\varphi(\xi_t)-\varphi(\xi_u)-\int_u^t(\frac{1}{2}\Delta\varphi(\xi_s)+b(s,\xi_s).\nabla\varphi(\xi_s))ds\right)\psi(\xi_{s_1},\cdots,\xi_{s_p}).$$
We are going to check in the last step of the proof that $\lim_{h\to 0}\E[F(X^h)]=0$. Unfortunately, the lack of continuity of the functional $F$ on ${\mathcal C}$ prevents from deducing immediately that $\int_{\cal C}F(\xi)P(d\xi)=0$. That is why we introduce for $\varepsilon\in(0,1]$,  a smooth and bounded function $b_\varepsilon $  approximating the original drift $b$ in \eqref{eds} such that, setting $b_\varepsilon^K(t,x)=\I_{[-K,K]^d}(x)b_\varepsilon(t,x)$ and $b^K(t,x)=\I_{[-K,K]^d}(x)b(t,x)$ for $K\in\N^*$, 
\begin{align}
                                                              &\forall K\in\N^*,\;\|b^K_\varepsilon-b^K\|_{L^{\tilde q}-L^{\tilde \rho}} \underset{\varepsilon \rightarrow 0}{\longrightarrow} 0\label{convbepsb}\\
 &\mbox{ with }(\tilde \rho,\tilde q)=\begin{cases}
     (\rho,q)\mbox{ if }\rho<\infty \mbox{ and }q<\infty,
     \\
      (\rho,\frac{2\rho +1}{\rho-d})\mbox{ if }\rho<\infty\mbox{ and }q=\infty,\\
   (\frac{dq+1}{q-2},q)\mbox{ if }\rho=\infty \mbox{ and }q<\infty
     .
   \end{cases}\notag\end{align}Note that $\frac{d}{\tilde \rho}+\frac{2}{\tilde q}<1$.
The functional $F_\varepsilon$ defined like $F$ but with $b_\varepsilon$ replacing $b$ is continuous and bounded and therefore, for fixed $\varepsilon\in (0,1]$, $\int_{\cal C}F_\varepsilon(\xi)P(d\xi)=\lim_{h\to 0}\E[F_\varepsilon(X^h)]=\lim_{h\to 0}\E[F_\varepsilon(X^h)-F(X^h)]$. We deduce that
$$\left|\int_{\cal C}F(\xi)P(d\xi)\right|\le \limsup_{\varepsilon\to 0}\int_{\cal C}|F(\xi)-F_\varepsilon(\xi)|P(d\xi)+\limsup_{\varepsilon\to 0}\limsup_{h\to 0}\E[|F_\varepsilon(X^h)-F(X^h)|].$$
Let $K\in\N^*$ be such that $\varphi$ vanishes outside $[-K,K]^d$. One has, using \eqref{estigausschem} then H\"older's inequality in space together with $\|g_c(s,.)\|_{L^{\frac{\tilde\rho}{\tilde\rho-1}}}=\left(\frac{\tilde\rho}{\tilde\rho-1}\right)^{-\frac{d(\tilde\rho-1)}{2\tilde\rho}}(2\pi c s)^{-\frac d {2\tilde\rho}}$ and last H\"older's inequality in time,
\begin{align*}
   \E[|F_\varepsilon(X^{h})-F(X^{h})|]&\le \|\psi\|_{L^\infty}\|\nabla\varphi\|_{L^\infty}\int_u^t\E[|b^K_\varepsilon(s,X^{h}_s)-b^K(s,X^{h}_s)
  |]ds\notag\\
                                      &\le C\|\psi\|_{L^\infty}\|\nabla\varphi\|_{L^\infty}\int_u^t\int_{\R^d}|b^K_\varepsilon(s,y)-b^K(s,y)|g_c(s,y-x)dyds\notag\\
                                      &\le C\|\psi\|_{L^\infty}\|\nabla\varphi\|_{L^\infty}\int_u^t\frac{\|b^K_\varepsilon(s,.) -b^k(s,.)\|_{L^{\tilde \rho}}ds}{s^{\frac{d}{2\tilde \rho}}}\\&\le C\|\psi\|_{L^\infty}\|\nabla\varphi\|_{L^\infty}\|b^K_\varepsilon-b^k\|_{L^{\tilde q}-L^{\tilde \rho}}(t-u)^{1-\left(\frac 1 {\tilde q}+\frac d{2\tilde \rho}\right)}.\end{align*}
                                    Since the same estimation holds for $\int_{\cal C}|F(\xi)-F_\varepsilon(\xi)|P(d\xi)$, we conclude that $\int_{\cal C}F(\xi)P(d\xi)$ $=0$. Taking $\varphi,\psi,u,s_1,\ldots,s_p,t$ in countable dense subsets, we deduce that $P$ solves the martingale problem associated with the stochastic differential equation
$$X_t=x+W_t+\int_0^tb(s,X_s)dr,\;t\in[0,T].$$
Since by \cite{kryl:rock:05}, existence of a pathwise unique strong solution holds for this equation, $P$ is the distribution of the solution. 

To check that $\lim_{h\to 0}\E[F(X^h)]=0$, we compute $\varphi(X^h_t)-\varphi(X^h_u)$ by It\^o's formula and take expectations to obtain that
\begin{align*}
  \E[F(X^h)]&=\E\left[\left(\int_u^t\left(b_h\left(U_{\lfloor \frac s h\rfloor},X^h_{\tau^h_s}\right)-b(s,X^h_{s})\right).\nabla\varphi(X^h_{s})ds \right)\psi(X_{s_1}^h,\cdots,X_{s_p}^h)\right]\\
  &=\Delta_1+\Delta_2+\Delta_3+\Delta_4\mbox{ with }\\
  \Delta_1&=\E\left[\left(\int_u^tb_h\left(U_{\lfloor \frac s h\rfloor},X^h_{\tau^h_s}\right).\left(\nabla\varphi(X^h_{s})-\nabla\varphi(X^h_{\tau^h_s})\right)ds\right)\psi(X_{s_1}^h,\cdots,X_{s_p}^h)\right],\\
  \Delta_2&=\E\left[\left(\int_u^t\left(b_h\left(U_{\lfloor \frac s h\rfloor},X^h_{\tau^h_s}\right)-b_h(s,X^h_{\tau^h_s})\right).\nabla\varphi(X^h_{\tau^h_s})ds\right)\psi(X_{s_1}^h,\cdots,X_{s_p}^h)\right],\\
  \Delta_3&=\E\left[\left(\int_u^t\left(b_h(s,X^h_{\tau^h_s})-b(s,X^h_{\tau^h_s})\right).\nabla\varphi(X^h_{\tau^h_s})ds\right)\psi(X_{s_1}^h,\cdots,X_{s_p}^h)\right],\\\Delta_4&=\E\left[\left(\int_u^t\left(b(s,X^h_{\tau^h_s}).\nabla\varphi(X^h_{\tau^h_s})-b(s,X^h_{s}).\nabla\varphi(X^h_{s})\right)ds\right)\psi(X_{s_1}^h,\cdots,X_{s_p}^h)\right].
\end{align*}
By the Cauchy-Schwarz inequality and \eqref{majocar},
\begin{align*}
  |\Delta_1|&\le C\|\psi\|_{L^\infty}\left(\int_u^t\E\left|\nabla\varphi(X^h_{s})-\nabla\varphi(X^h_{\tau^h_s})\right|^2ds\right)^{\frac{1}{2}}\\
  &\le C\|\psi\|_{L^\infty}\|\nabla^2\varphi\|_{L^\infty}\left(\int_u^t\E\left[|W_s-W_{\tau^h_s}|^2+(s-\tau^h_s)\int_{\tau^h_s}^s\left|b_h\left(U_{\lfloor \frac r h\rfloor},X^h_{\tau^h_r}\right)\right|^2 dr\right]ds\right)^{\frac{1}{2}}\\
  &\le C h^{\frac 1 2}.
\end{align*}
Since $\E\left[\int_{\lceil\frac u h\rceil h}^{\tau^h_t}\left(b_h\left(U_{\lfloor \frac s h\rfloor},X^h_{\tau^h_s}\right)-b_h(s,X^h_{\tau^h_s})\right).\nabla\varphi(X^h_{\tau^h_s})ds| \sigma(X_r^h,\ r\in [0,\lceil\frac u h\rceil h])
\right]=0$ and $$\E\left[\int_0^T\left|b_h\left(s,X^h_{\tau^h_s}\right)\right|^2ds\right]=\E\left[\int_0^T\left|b_h\left(U_{\lfloor \frac s h\rfloor},X^h_{\tau^h_s}\right)\right|^2ds\right],$$ we have, using \eqref{majocar} and the Cauchy-Schwarz inequality for the second inequality
\begin{align*}
   |\Delta_2|\le& \|\psi\|_{L^\infty}\|\nabla\varphi\|_{L^\infty}\E\left[\int_u^{\lceil\frac u h\rceil h}\left|b_h\left(U_{\lfloor \frac s h\rfloor},X^h_{\tau^h_s}\right)\right|+|b_h(s,X^h_{\tau^h_s})|ds\right.\\
   &\hspace*{1cm}\left.+\int_{\tau^h_t}^t\left|b_h\left(U_{\lfloor \frac s h\rfloor},X^h_{\tau^h_s}\right)\right|+|b_h(s,X^h_{\tau^h_s})|ds\right]
   \le C h^{\frac 1 2}.
\end{align*}
Since $|b_h-b|\le |b|\I_{\{|b|\ge B h^{-\left(\frac 1 q+\frac d{2\rho}\right)}\}}\le B^{-1}h^{\frac 1 q+\frac d{2\rho}}|b|^2
$, we have, using \eqref{majocar} for the last inequality,
\begin{align*}
   \forall h\le u,\;|\Delta_3|\le B^{-1}h^{\frac 1 q+\frac d{2\rho}}\|\psi\|_{L^\infty}\|\nabla\varphi\|_{L^\infty}\E\left[\int_h^T\left|b\left(s,X^h_{\tau^h_s}\right)\right|^2ds\right]\le Ch^{\frac 1 q+\frac d{2\rho}}.
\end{align*}

Last, for $h\le u$, using \eqref{estischemtemps} and that $\tau^h_s\ge \frac{s}{2}$ for $s\ge h$ for the second inequality, then $\|g_c(s,.)\|_{L^{\bar \rho}}={\bar \rho}^{-\frac{d}{2\bar\rho}}(2\pi c s)^{-(1-\frac 1{\bar \rho})\frac d 2}$ with $\frac 1{\bar\rho}=1-\frac 1{\rho}$, H\"older's inequality in space and in time, we obtain that

\begin{align*}
  |\Delta_4|\le& \|\psi\|_{L^\infty} \|\nabla\varphi\|_{L^\infty}\int_u^t\int_{\R^d}|\Gamma^h(0,x,\tau^h_s,y)-\Gamma^h(0,x,s,y)||b(s,y)|dy ds\\
  \le& C\int_u^t\int_{\R^d}\frac{(s-\tau^h_s)^{\frac{\alpha}{2}}}{s^{\frac{\alpha}{2}}}g_c(s,y-x)|b(s,y)|dy ds\le C h^{\frac{\alpha}{2}}\int_u^t\frac{\|b(s,.)\|_{L^\rho}ds}{s^{\frac{\alpha}{2}+\frac d{2\rho}}}\\
  \le& C h^{\frac{\alpha}{2}}\|b\|_{L^q-L^\rho}(t-u)^{\frac 12}.
\end{align*}
We conclude that $\lim_{h\to 0}\E[F(X^h)]=0$.

\appendix

Let us eventually check that the  Duhamel representation \eqref{DUHAMEL_DIFF_PROP} for the density of the diffusion holds. This can be done by reasoning like in the above derivation of \eqref{DUHAMEL_SCHEME_PROP} in Subsection \ref{densduh} and using that by \eqref{CTR_GRAD_GAUSS}, \eqref{HK_AND_GRAD_SPACE_DENS} and Lemma \ref{GAUSS_INT_LEMMA},
\begin{align*}
  &\int_{0}^{ t}\E\left[\left|b(r,X_{r})\cdot\nabla_y g_1(t-r,y-X_r)\right|\right]dr\\
  \le& C\int_0^t\int_{\R^d}g_c(r,z-x)|b(r,z)|\frac{g_c(t-r,y-z)}{(t-r)^{\frac 12}}dzdr\\
  \le &C\|b\|_{L^q-L^\rho}g_c(t,y-x)t^{\frac{\alpha}{2}}<+\infty.
\end{align*}

\section{Proof of Lemma \ref{LEM_GAUS_SENS}: usual Gaussian estimates}\label{SEC_PROOF_USUAL_GAUSS_EST}
\begin{proof}
The bound \eqref{CTR_GRAD_GAUSS} is standard. If $|x-x'|\ge u^{\frac 12} $ then \eqref{DIFF_GRAD_GAUSS_LEM} precisely follows from the first inequality in \eqref{CTR_GRAD_GAUSS}. Namely,
$$\Big|\nabla^\zeta _x {\tildeGamma}(u,x)-\nabla^\zeta _x {\tildeGamma}(u,x')\Big|\le\big|\nabla^\zeta_x {\tildeGamma}(u,x)\big|+\big|\nabla^\zeta_x {\tildeGamma}(u,x')\big|\le \frac C{u^{\frac{|\zeta|}2}} \big(g_c(u,x)+g_c(u,x')\big).$$Assume now $|x-x'|\le u^{1/2} $. Write
\begin{align*}
\nabla^\zeta_x {\tildeGamma}(u,x')-\nabla^\zeta_x {\tildeGamma}(u,x)=\int_0^1 d\lambda \nabla_x \nabla^\zeta_x{\tildeGamma}(u,x+\lambda(x'-x))\cdot(x'-x).
\end{align*}
Observe now that  from the first inequality in \eqref{CTR_GRAD_GAUSS} applied with $c$ and $\zeta$ replaced by $\frac{1+c}{2}>1$ and the sum of $\zeta$ and a vector in the canonical basis of $\R^d$ (note that \eqref{CTR_GRAD_GAUSS} remains valid for multi-indices with length bounded from above by $3$) :
\begin{align}\label{AFTER_TAYL_G}
|\nabla_x^\zeta {\tildeGamma}(u,x')-\nabla_x^\zeta {\tildeGamma}(t-s,x)|\le C\frac{|x-x'|}{u^{\frac {d+1+|\zeta|}{2}}} \int_0^1 d\lambda \exp\left(-\frac{|x+\lambda(x'-x)|^2}{(1+c)u} \right).
\end{align}
Recall now that since $\lambda \in [0,1] $ and that  $|x-x'|\le u^{\frac 12} $:
\begin{align*}
|x+\lambda(x'-x)|^2\ge &\frac{1+c}{2c}|x|^2-\frac{c+1}{c-1}\lambda^2|x-x'|^2\ge \frac{1+c}{2c}|x|^2-\frac{c+1}{c-1}|x-x'|^2\\
\ge& \frac{1+c}{2c}|x|^2-\frac{c+1}{c-1}u,
\end{align*}
which plugged into \eqref{AFTER_TAYL_G} yields:
\begin{eqnarray*}
|\nabla_x^\zeta {\tildeGamma}(u,x)-\nabla_x^\zeta {\tildeGamma}(u,x')|&\le& C\frac{|x-x'|}{u^{\frac {d+1+|\zeta|}{2}}}  \exp\left(-\frac{|x|^2}{2cu} \right)\exp\left(\frac{1}{c-1}\right),
\end{eqnarray*}
which, up to a modification of the constant $C$, concludes the proof of \eqref{DIFF_GRAD_GAUSS_LEM}.
The bound \eqref{DIFF_GRAD_GAUSS_LEM_TIME} is obtained the same way by bounding from above each term of the difference using the first inequality in \eqref{CTR_GRAD_GAUSS} when $|u'-u|\ge u$ and by integrating the second inequality in \eqref{CTR_GRAD_GAUSS} and using that for $v\in[u,u']$, $\frac{g_c(v,x)}{v^{1+\frac{|\zeta|}2}}\le\frac{u'^{\frac{d}{2}}g_c(u',x)}{u^{\frac{d}{2}+1+\frac{|\zeta|}2}}\le 2^{\frac{d}{2}}\frac{g_c(u',x)}{u^{1+\frac{|\zeta|}2}}$ when $|u'-u|\le u$.
\end{proof}

\section{Proof of Lemma \ref{GR_VOL_LEMMA}}\label{APP_GR}
\begin{proof}{\bf (i)}
Iterating \eqref{sansdeux} $n\ge 1$ times, we get that for all $t\in[0,T]$,
\begin{align*}
    f(t)\le &\eta\left(1+\sum_{k=1}^{n-1}\delta^kt^\beta\int_0^ts_1^{\beta-\tilde\beta}\int_0^{s_1}s_2^{\beta-\tilde\beta}\hdots\int_0^{s_{k-2}}s_{k-1}^{\beta-\tilde\beta}\int_0^{s_{k-1}}\frac{ds_k}{s_k^{\tilde\beta}}ds_{k-1}\hdots ds_1\right)\\&+\delta^nt^\beta\int_0^ts_1^{\beta-\tilde\beta}\int_0^{s_1}s_2^{\beta-\tilde\beta}\hdots\int_0^{s_{n-2}}s_{n-1}^{\beta-\tilde\beta}\int_0^{s_{n-1}}\frac{f(s_n)ds_n}{s_n^{\tilde\beta}}ds_{n-1}\hdots ds_1
\end{align*}
The term with index $k$ in the sum is equal to
$\frac{\delta^k t^{k(1+\beta-\tilde\beta)}}{\prod_{j=1}^k(j(1+\beta-\tilde\beta)-\beta)}$ so that the first line in the right-hand side is not greater than the finite constant not depending on $f$ \begin{align*}
   C_{\beta,\tilde \beta,\eta,\delta,T}=&\eta+\eta\sum_{k\ge 1}\frac{\delta^k T^{k(1+\beta-\tilde\beta)}}{\prod_{j=1}^k(j(1+\beta-\tilde\beta)-\beta)}\le \eta+\eta\sum_{k\ge 1}\frac{\delta^k T^{k(1+\beta-\tilde\beta)}}{\prod_{j=1}^kj(1+\beta\wedge 0-\tilde\beta)}\\=&\eta e^{\frac{\delta T^{1+\beta-\tilde\beta}}{1+\beta\wedge 0-\tilde\beta}} .
\end{align*}On the other hand, the last term in the right-hand side is bounded from above by $\frac{\delta^n t^{n(1+\beta-\tilde\beta)}}{\prod_{j=1}^n(j(1+\beta-\tilde\beta)-\beta)}\sup_{s\in[0,t]}f(s)$ and converges to $0$ as $n\to\infty$.

{\bf (ii)} If $\hat\beta\le 0$, then \eqref{sansdeux} holds with $\eta=a$, $\delta=b$ and $\beta=\check \beta-\hat\beta$ which is larger than $\tilde\beta-1$ so that the conclusion follows by {\bf (i)}. Let us suppose that $\hat\beta> 0$ and check that by iterating the inequality \eqref{avecdeux}, we obtain \eqref{sansdeux} for some finite constants $\eta,\delta$ and some $\beta>\tilde\beta-1$. Iterating \eqref{avecdeux} once, we obtain
  \begin{align}
    \forall t\in[0,T],\;f(t)&\le a+ab t^{\check{\beta}}\int_0^t\frac{ds}{s^{\tilde \beta}(t-s)^{\hat\beta}}+b^2t^{\check{\beta}}\int_0^t\frac{1}{s^{\tilde \beta-\check{\beta}}(t-s)^{\hat\beta}}\int_0^s\frac{f(u)du}{u^{\tilde \beta}(s-u)^{\hat\beta}}ds.\label{premit}\end{align}
  We now  set $\gamma:= 1+\check{\beta}-\tilde\beta-\hat\beta>0$ and distinguish two situations depending on the sign of $\tilde \beta-\check{\beta}$. Let us first suppose that $\tilde \beta-\check{\beta}\ge 0$. Using Fubini's theorem and the inequality $s^{\tilde \beta-\check{\beta}}\ge (s-u)^{\tilde \beta-\check{\beta}}$, we obtain
  \begin{align*}
    f(t)&\le a+ab T^{\gamma}B(1-\tilde\beta,1-\hat\beta)+b^2t^{\check{\beta}}\int_0^t\frac{f(u)}{u^{\tilde \beta}}\int_u^t\frac{ds}{(s-u)^{\hat\beta+\tilde \beta-\check{\beta}}(t-s)^{\hat\beta}}du\\
        &=a_1+b_1t^{\check{\beta}}\int_0^t\frac{f(u)du}{u^{\tilde \beta}(t-u)^{\hat\beta-\gamma}}
  \end{align*}
  with $a_1=a+ab T^{\gamma}B(1-\tilde\beta,1-\hat\beta)$ and $b_1=b^2B(\gamma,1-\hat\beta)$.
  In comparison with \eqref{avecdeux}, after this first step, the power of the last factor in the denominator has decreased from $\hat\beta$ to $\hat\beta-\gamma$. We may now at each step iterate the inequality obtained from the previous step and obtain after $n$ steps
  $$\forall t\in[0,T],\;f(t)\le a_n+b_nt^{\check{\beta}}\int_0^t\frac{f(s)ds}{s^{\tilde \beta}(t-s)^{\hat\beta-(2^n-1)\gamma}}$$
  where $a_n=a_{n-1}+a_{n-1}b_{n-1}T^{2^{n-1}\gamma}B(1\!-\!\tilde\beta,1+(2^{n-1}\!-\!1)\gamma\!-\!\hat\beta)$ and $b_{n}=b_{n-1}^2B(2^{n-1}\gamma,1+(2^{n-1}-1)\gamma-\hat\beta)$.
For $\hat n=\lceil\frac{\ln(1+\hat\beta/\gamma)}{\ln 2}\rceil$, $\hat\beta-(2^{\hat n}-1)\gamma\le 0$, so that \eqref{sansdeux} holds for $\eta=a_{\hat n}$ and $\delta=b_{\hat n}$ and $$\beta=\check\beta+(2^{\hat n}-1)\gamma-\hat\beta=2^{\hat n}\gamma+\tilde\beta-1>\tilde\beta-1\mbox{ since }\hat n\ge 1\mbox{ as }\hat\beta>0.$$

Let us now suppose that $\tilde\beta-\check\beta<0$. Inserting the inequality $s^{\check\beta-\tilde\beta}\le t^{\check\beta-\tilde\beta}$ in \eqref{premit}, we get
\begin{align*}
   f(t)&\le a+ab T^\gamma B(1-\tilde\beta,1-\hat\beta)+b^2t^{2\check{\beta}-\tilde\beta}\int_0^t\frac{f(u)}{u^{\tilde \beta}}\int_u^t\frac{ds}{(s-u)^{\hat\beta}(t-s)^{\hat\beta}}du\\&= a_1+b_1t^{2\check{\beta}-\tilde\beta}\int_0^t\frac{f(u)du}{u^{\tilde \beta}(t-u)^{2\hat\beta-1}}
\end{align*}
with $a_1=a+ab T^{\gamma}B(1-\tilde\beta,1-\hat\beta)$ and $b_1=b^2B(1-\hat\beta,1-\hat\beta)$.
In comparison with \eqref{avecdeux}, after this first step, the power of $t$ has increased from $\check\beta$ to $2\check\beta-\tilde\beta$ and the power in the last factor of the denominator has decreased from $\hat\beta$ to $2\hat\beta -1$.
We may now at each step iterate the inequality obtained from the previous step and obtain after $n$ steps
  $$\forall t\in[0,T],\;f(t)\le a_n+b_nt^{2^n(\check{\beta}-\tilde\beta)+\tilde\beta}\int_0^t\frac{f(s)ds}{s^{\tilde \beta}(t-s)^{2^n(\hat\beta-1)+1}}$$
  where $a_n=a_{n-1}+a_{n-1}b_{n-1}T^{2^{n-1}\gamma}B(1-\tilde\beta,2^{n-1}(1-\hat\beta))$ and $b_{n}=b_{n-1}^2B(2^{n-1}(1-\hat\beta),2^{n-1}(1-\hat\beta))$. For $\hat n=\lceil\frac{-\ln(1-\hat\beta)}{\ln 2}\rceil$, $2^{\hat n}(\hat\beta-1)+1\le 0$, so that \eqref{sansdeux} holds for $\eta=a_{\hat n}$ and $\delta=b_{\hat n}$ and $$\beta=2^{\hat n}(\check{\beta}-\tilde\beta)+\tilde\beta+2^{\hat n}(1-\hat\beta)-1=2^{\hat n}\gamma+\tilde\beta-1>\tilde\beta-1\mbox{ since }\hat n\ge 1\mbox{ as }\hat\beta>0.$$
\end{proof}

\section*{Acknowledgments.}
For the second author, the article was prepared within the framework of the HSE University Basic Research Program.

\end{document}